\documentclass[11pt]{amsart}

\usepackage{epigamath}

\usepackage[english]{babel}

\numberwithin{equation}{section}

\usepackage{etex} 

\usepackage{tensor}
\usepackage[all]{xy}

\usepackage{amsmath, amssymb, amsfonts, latexsym, mdwlist, amsthm}
\usepackage{subfig}
\usepackage{graphicx}
\usepackage{longtable}
\usepackage[shortlabels]{enumitem}
\setlist[enumerate,1]{label={\rm(\roman*)}, ref={\rm\roman*}}
\usepackage{mathtools} 

\usepackage{multirow}
\usepackage{diagbox}
\usepackage{arydshln}
\usepackage{bbold}
\usepackage{fix-cm} 

\usepackage[utf8x]{inputenc}

\usepackage{xfrac}    
\usepackage{faktor}

\newtheorem{theorem}{Theorem}[section]
\newtheorem{lemma}[theorem]{Lemma}
\newtheorem{corollary}[theorem]{Corollary}
\newtheorem{conj}[theorem]{Conjecture}
\newtheorem{proposition}[theorem]{Proposition}
\newtheorem{prop}[theorem]{Proposition}

\newtheorem*{theorem*}{Theorem}
\newtheorem*{lemma*}{Lemma}
\newtheorem*{main-theo}{Main Theorem}

\theoremstyle{definition}
\newtheorem{defn}[theorem]{Definition}

\theoremstyle{remark}
\newtheorem{remark}[theorem]{Remark}

\DeclareFontFamily{T1}{calligra}{}
\DeclareFontShape{T1}{calligra}{m}{n}{<-> s * [1.80] callig15}{}
\DeclareMathAlphabet{\mathcalligra}{T1}{calligra}{m}{n}
\DeclareFontFamily{OT1}{pzc}{}
\DeclareFontShape{OT1}{pzc}{m}{it}{<-> s * [1.30] pzcmi7t}{}
\DeclareMathAlphabet{\mathpzc}{OT1}{pzc}{m}{it}
\DeclareMathAlphabet{\pazocal}{OMS}{zplm}{m}{n}

\usepackage{tikz}
\usetikzlibrary{calc}

\tikzset{point/.style = {fill=gray,circle,inner sep=2pt}}

\newcommand\T{\rule{0pt}{2.6ex}} 
\newcommand\B{\rule[-1.2ex]{0pt}{0pt}} 
\newcommand\BStrut[1]{\rule[-#1ex]{0pt}{0pt}} 
\newcommand\TopStrut[1]{\rule{0pt}{#1ex}} 

\newcommand{\rvline}{\hspace*{-\arraycolsep}\vline\hspace*{-\arraycolsep}}

\newcommand\mU{{\boldsymbol{\mu}}}

\newcommand\CC{{\mathbb C}}

\newcommand\PP{{\mathbb P}}
\newcommand\QQ{{\mathbb Q}}
\newcommand\RR{{\mathbb R}}
\newcommand\TT{{\mathbb T}}

\newcommand\ZZ{{\mathbb Z}}

\newcommand\JJJ{{\pazocal J}}
\newcommand\LLL{{\mathcal L}}

\newcommand\OOO{{\mathcal O}}

\newcommand\SSS{{\mathcal S}}

\newcommand\HS{\includegraphics{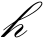}\!}
\newcommand\HSer{{\mathpzc h}}

\newcommand{\Cl}{\operatorname{Cl}\nolimits}

\newcommand{\diag}{\operatorname{diag}\nolimits}

\newcommand{\Ext}{\operatorname{Ext}\nolimits}

\newcommand{\id}{\operatorname{id}\nolimits}

\renewcommand{\Im}{\operatorname{Im}\nolimits}

\newcommand\Proj{\operatorname{Proj}\nolimits}

\newcommand\Sing{\operatorname{Sing}\nolimits}
\newcommand\Sp{\boldsymbol{Sp}}

\newcommand{\tr}{\operatorname{tr}\nolimits}

\newcommand{\ud}{\hspace{0.3ex}\mathrm{d}\hspace{-0.1ex}}
\renewcommand\mod{{\,\mathrm{mod}\,}}
\newcommand\dual{{}^{\scriptscriptstyle\vee}}
\renewcommand\tilde[1]{\widetilde{#1}}
\renewcommand\bar[1]{\overline{#1}}

\def\trans#1{{#1}^{{\sf T}}}
\newcommand\thch[2]{\raisebox{2pt}{$\genfrac{[}{]}{0pt}{2}{#1}{#2}$}}
\newcommand\pd{\partial}

\newcommand\lra{\longrightarrow}

\newlength{\rrrr}

\newcommand{\intoo}[1]{\:
\xymatrix@1{\ar@{^(->}[r]^{#1}&}\:}

\newcommand{\ootni}[1]{\:
\xymatrix@1{&\ar@{_(->}[l]_(.3){#1}}\:}

\newcommand{\loeqs}[1]{\:
\xymatrix@1{\ar@{=}[rr]^{\ #1\ }&&}\:}

\newcommand{\verteq}{\rotatebox{90}{$\,=$}}

\def\geq{\geqslant}
\def\leq{\leqslant}

\newcommand{\dss}{\displaystyle}

\newcommand{\supth}[1]{\ensuremath{#1^{\mathrm{th}}}}

\EpigaVolumeYear{8}{2024} \EpigaArticleNr{9} \ReceivedOn{June 28, 2023}
\InFinalFormOn{December 3, 2023}
\AcceptedOn{February 7, 2024}

\title{Action of the automorphism group on the Jacobian of Klein's quartic curve II: Invariant theta functions}
\titlemark{Action of the automorphism group on the Jacobian of Klein's quartic curve II}

\author{Dimitri Markushevich} 
\address{Univ. Lille, CNRS, UMR 8524 -- Laboratoire Paul Painlev\'e, 59000 Lille, France}
\email{dimitri.markouchevitch@univ-lille.fr}

\author{Anne Moreau}
\address{Universit\'{e} Paris-Saclay, CNRS, Laboratoire de Math\'{e}matiques d'Orsay, Rue Michel Magat, B\^{a}t.~307, 91405 Orsay, France}
\email{anne.moreau@universite-paris-saclay.fr}

\authormark{D.~Markushevich and A.~Moreau}

\AbstractInEnglish{Bernstein--Schwarzman conjectured that the quotient of a complex affine space by an irreducible complex crystallographic group generated by reflections is a weighted projective space.  The conjecture was proved by Schwarzman and Tokunaga--Yoshida in dimension~$2$ for almost all such groups, and for all crystallographic reflection groups of Coxeter type by Looijenga, Bernstein--Schwarzman and Kac--Peterson in any dimension.  We prove that the conjecture is true for the crystallographic reflection group in dimension $3$ for which the associated collineation group is Klein's simple group of order $168$.  In this case, the quotient is the $3$-dimensional weighted projective space with weights $1, 2, 4, 7$.  The main ingredient in the proof is the computation of the algebra of invariant theta functions.  Unlike in the Coxeter case, the invariant algebra is not free polynomial, and this was the major stumbling block.}



\acknowledgement{D.M. was partially supported by the PRCI SMAGP (ANR-20-CE40-0026-01) and the Labex CEMPI  (ANR-11-LABX-0007-01). 
A.M. is partially supported by the European Research Council (ERC) under the European Union's Horizon 2020 research innovation programme (ERC-2020-SyG-854361-HyperK).}

\begin{document}



\maketitle

\begin{prelims}

\DisplayAbstractInEnglish

\bigskip

\DisplayKeyWords

\medskip

\DisplayMSCclass







\end{prelims}


\newpage

\setcounter{tocdepth}{1}

\tableofcontents


\section{Introduction}
A general conjecture of Bernstein and Schwarzman~\cite{Be-Sch3} claims that the quotient of $\CC^n$ by the action of an irreducible complex crystallographic group generated by reflections (complex crystallographic reflection group for short) is a weighted projective space.  The conjecture is proved for almost all complex crystallographic reflection groups for $n=2$, as well as for irreducible complex crystallographic reflection groups of Coxeter type of any rank $n \geq 2$.  A complex crystallographic reflection group is said to be {\em of Coxeter type} if it is obtained by complexification of a real crystallographic group or, in other words, if the group of its linear parts is conjugate to a finite subgroup of the orthogonal group $\mathbf O(n)$. See~\cite{G336-I} for more historical comments and further references.

Since the 1980s, the conjecture remained widely open for any irreducible complex crystallographic reflection group of rank $n\geq 3$ which is genuinely complex, that is, not of Coxeter type.  In the present paper we prove the conjecture for the rank $3$ complex crystallographic reflection group $\Gamma$ of type $[K_{24}]$ in the classification of Popov~\cite{Po}.  We show that the quotient $X=\CC^3/\Gamma$ is isomorphic to a weighted projective space.\footnote{The proof of the Bernstein--Schwarzman conjecture for $[K_{24}]$ appeared in the second version of the present article, posted on arXiv in November 2022, and at that moment it was the only non-Coxeter complex crystallographic reflection group of rank greater than~$2$ for which the conjecture was established. In March 2023, Eric Rains posted the preprint~\cite{Rains}, in which he proves the conjecture in full generality.}  More explicitly, the main result of the article is the following.

\begin{main-theo}[Theorem~\ref{main3}]
The quotient variety $\JJJ/G$, where $\JJJ=\JJJ(C)$  is the $3$-dimensional Jacobian  of the plane Klein quartic curve $C$ and $G$ is the full automorphism group of order $336$, is isomorphic to the weighted projective space $\PP(1,2,4,7)$.
\end{main-theo}

The  group $[K_{24}]$ is in several regards the most intriguing one among the rank $3$ complex crystallographic reflection groups.
Firstly, it is the only one whose projectivized group of linear parts is simple, namely, is equal to Klein's group $H$ of order $168$. The groups of linear parts in all the other cases are solvable. 
Secondly, the quotient $\CC^3/\Gamma$ 
is isomorphic to the quotient of the Jacobian $\JJJ(C)$ of  
Klein's quartic curve 
$$C:=\{x^3y+y^3z+z^3x=0\}\subset \PP^2$$ by the full group 
$$G=\{\pm1\}\times H$$ of its automorphisms as a principally polarized abelian variety.  As follows from Hurwitz' bound, $C$ has a maximal number of automorphisms for a curve of genus $3$. Algebraic varieties acted on by Klein's group, not only curves, are a recurrent subject of interest in algebraic geometry. Thirdly, $[K_{24}]$ encloses a very rich number-theoretic content, for Klein's curve $C$ is nothing but the modular curve $X(7)$, or else it can be viewed as a Shimura curve, while its Jacobian is isomorphic, as an abstract abelian variety, to the cube of the elliptic modular curve $X_0(49)$.  The representation of the translation lattice of $\Gamma$ that we use was given in~\cite{Mazur}.

An important ingredient of our proof for $[K_{24}]$ is the computation of the Hilbert function of the algebra of $\Gamma$-invariant theta functions.
The proofs of the conjecture in the Coxeter case obtained in the 1980s passed through showing that the invariant algebra is free. As the quotient variety is the spectrum of the invariant algebra, the freeness of the latter implies that the quotient is a weighted projective space.  But this idea does not work for genuinely complex crystallographic reflection groups since the invariant algebra is not free anymore. In the case of $[K_{24}]$, we succeed to understand the structure of relations between generators of this algebra. It turns out that the ideal of relations is principal and defines a hypersurface of degree $8$ in the $4$-dimensional weighted projective space $\PP(1,1,2,4,7)$.

On the other hand, in~\cite{G336-I}, we determined the singularities of $X$: they are images of the orbits whose stabilizers are not generated by reflections.  We observed that the singularities of $X$ are analytically equivalent to those of $\PP(1,2,4,7)$, which prompted us to conjecture that $X$ is isomorphic to this weighted projective space.  The latter also embeds in $\PP(1,1,2,4,7)$ as an octic hypersurface! It can be given by the equation $y_0y_4=y_3^2.$ The last step of our proof is to show that all degree $8$
hypersurfaces in $\PP(1,1,2,4,7)$ whose singularities are those of $\PP(1,2,4,7)$ are equivalent to $y_0y_4=y_3^2$ by automorphisms of $\PP(1,1,2,4,7)$.

As a byproduct, we see that $\PP(1,2,4,7)$ possesses a non-trivial deformation, obtained by varying the coefficients of the octic in the $4$-dimensional weighted projective space. This deformation turns out to be versal and is a partial smoothing, so that the general member of the deformation family is a 2-Gorenstein Fano 3-fold with Picard group $\ZZ$ and only rigid isolated singularities. Deformations of 1-Gorenstein weighted projective spaces in dimension~$3$ have been studied in~\cite{DS}.

Another noteworthy feature of our quotient $X$ is the existence of a natural double cover $\bar Y\to X$, a Calabi--Yau orbifold that can be used as a target space for the superstring compactification; in particular, an interesting question is what its mirror family is. This double cover is obtained as the quotient $\bar Y=\CC^3/\Gamma_0$, where $\Gamma_0$ is the subgroup of index $2$ in $\Gamma$, the unimodular part of $\Gamma$. It can also be realized as an anticanonical hypersurface in $\PP(1,2,4,7,14)$ of dimension $4$. However, the known procedure of constructing mirrors of generic hypersurfaces in a weighted projective space does not apply to this case, for $\bar Y$ is by no means generic; it is a very special member of the anticanonical system on this weighted projective space having non-isolated singularities.  We hope to return to the study of this Calabi--Yau orbifold in the future.

Let us now detail our strategy to prove the main theorem.  In general, a {\em complex crystallographic group} $\Gamma$ of rank $n$ is a group of affine transformations of $\CC^n$ which fits in an exact triple
$$
0\lra L\lra \Gamma\lra \ud\Gamma\lra 1,
$$
where $L\simeq \ZZ^{2n}$ is a lattice of maximal rank $2n$ in $\CC^n$, acting by translations, and $\ud\Gamma$ is a finite subgroup of the unitary group $\mathbf U(n)$.  A complex crystallographic group is a {\em complex crystallographic reflection group} if it is generated by complex affine reflections, where an affine transformation of $\CC^n$ is called a {\em reflection} if it is of finite order and its fixed locus is an affine hyperplane. In our case $\ud\Gamma=G$; there is a unique $G$-invariant rank~$6$ lattice $L$ in $\CC^3$, which turns out to be the period lattice $\Lambda$ of Klein's quartic $C$, and the above extension is necessarily split, so that $\Gamma=\Lambda\rtimes G$.

The quotient $\CC^n/\Gamma$ can be thought of as the quotient of the abelian variety $\CC^n/L=A$ by the induced action of the finite group of linear parts: $\CC^n/\Gamma\simeq A/\ud\Gamma$.  The idea applied in the case of irreducible complex crystallographic reflection groups of Coxeter type in~\cite{Be-Sch3,Loo} is to represent $A$ as the Proj of the graded algebra $S(\LLL)$ of sections of the powers of an ample line bundle $\LLL$ on $A$, linearizable by the action of $\ud\Gamma$, and then $\CC^n/\Gamma$ is the Proj of the $\ud\Gamma$-invariant part:
$$
S(\LLL)=\bigoplus_{k=0}^\infty H^0\left(A,\LLL^{k}\right),\quad A=\Proj S(\LLL),\quad \CC^n/\Gamma\simeq A/\ud\Gamma =\Proj \left(S(\LLL)^{\ud\Gamma}\right).
$$
The sections of the powers of $\LLL$ are given by theta functions for the lattice $L$, and it is proven that the invariant part is a polynomial algebra in $n+1$ free generators.

Mimicking this approach, we introduce the theta functions $\theta_{m,k}\colon \CC^3\to\CC$ of degree $k\geq 0$ for the lattice $\Lambda$ in such a way that $\{\theta_{m,k}\}_{m\in P_k}$ is a basis of the space $H^0(\JJJ,\LLL^k)$ of global sections of the $\supth{k}$ power of an appropriate line bundle $\LLL$ on $\JJJ=\CC^3/\Lambda$ defining a principal polarization, where $m$ runs over a set $P_k$ of representatives of $\frac1k\ZZ^3/\ZZ^3$.  Unlike in the Coxeter case, the thus defined line bundle $\LLL$ is not ${\ud\Gamma}$-invariant, only even powers of $\LLL$ can be $\ud\Gamma$-linearized, and thus we have to work with the second Veronese subalgebra $S(\LLL^2)$ of $S(\LLL)$. We have
$$
X=A/\ud\Gamma\simeq \Proj S\left(\LLL^2\right)^{\ud\Gamma},\quad S\left(\LLL^2\right)^{\ud\Gamma}=\bigoplus_{p=0}^\infty H^0\left(A,\LLL^{2p}\right)^{\ud\Gamma}.
$$
Thus one cannot expect $S(\LLL^2)^{\ud\Gamma}$ to be polynomial; what we prove is the isomorphism between $S(\LLL^2)^{\ud\Gamma}$ and the second Veronese algebra of $\PP(1,2,4,7)$, the latter being non-polynomial.

Now we describe the content of the paper by sections. Section~\ref{sec:prelim} gathers definitions and preliminary results on the complex crystallographic reflection group $\Gamma$, the lattice $\Lambda$ and associated theta functions.

In Section~\ref{sec:theta formula}, we compute the transformation formula for the action of the elements of the modular group on our theta functions (Theorem~\ref{TTF}).

In Section~\ref{sec:Unitary action}, we determine the Hilbert function of the algebra $S(\LLL^2)^G$, proving that it coincides with the Hilbert function of the second Veronese algebra of $\PP(1,2,4,7)$; see Theorem~\ref{main}.  This is the second step of the proof of the isomorphism $X\simeq \PP(1,2,4,7)$; the first one was the study of the singularities of $X$, done in~\cite[Theorem 4.3]{G336-I}.

Next, we prove that $X$ admits an embedding in $\PP(1,1,2,4,7)$ as a degree $8$ hypersurface. This is done in Section~\ref{sec:Main Theorem}.  We first show that there exist four homogeneous elements $\varphi_0,\ldots,\varphi_3$ of $S(\LLL^2)^G$ of degrees $1,1,2,4$ which are algebraically independent. The four invariant theta functions $\varphi_i$ are chosen in an \textit{ad hoc} way in Lemma~\ref{Jac-non-0}, and their algebraic independence is proved by evaluating their Jacobian.  Using the known Hilbert function, we deduce from this that there is a fifth element $\varphi_4$ of $S(\LLL^2)^G$ of degree $7$ such that the five functions $\varphi_i$ generate $S(\LLL^2)^G$, and this makes $X$ into a degree $8$ hypersurface in $\PP(1,1,2,4,7)$ (see Theorem~\ref{main2}).

Section~\ref{sec:degree-8 hypersurfaces} contains the last step of the proof of the main result: in Proposition~\ref{pro:normal forms}, we show that the degree~$8$ hypersurfaces in $\PP(1,1,2,4,7)$ whose singularities are those of $\PP(1,2,4,7)$ form just one orbit under the action of the group of coordinate changes in $\PP(1,1,2,4,7)$. Thus $X$ is equivalent, modulo a coordinate change, to $\PP(1,2,4,7)$ embedded in $\PP(1,1,2,4,7)$ as a degree $8$ hypersurface. This implies the main result of the paper (Theorem~\ref{main3}). In Remark~\ref{smoothing}, we discuss the non-trivial deformations of $\PP(1,2,4,7)$ smoothing out the non-isolated singularity, provided by the family of weighted octics in $\PP(1,1,2,4,7)$.

In Section~\ref{sec:various}, we look at the quotient $\CC^3/\Lambda\rtimes W$, where $W$ is a maximal real reflection subgroup of $G$, of order $48$. The group $\Lambda\rtimes W$ is complex crystallographic but is not generated by reflections. It is plausible that this quotient is a weak weighted projective space (see Proposition~\ref{fake}), and we formulate a conjecture, generalizing this example, which says that if $\Gamma$ and $\Gamma_1$ are commensurable complex crystallographic groups with the same linear parts, such that $\Gamma$ is complex crystallographic reflection, then $\CC^n/\Gamma_1$ is a weak weighted projective space (see Conjecture~\ref{WWPS}). We also provide in Section~\ref{sec:various} some heuristic explanation of our \textit{ad hoc} choice for generators of $S(\LLL^2)^G$ (see Remark~\ref{conj-gen}).

\section*{Acknowledgements}
The authors thank the referee for the useful comments.  They also wish to thank Thomas Dedieu and Xavier Roulleau for interesting discussions.

\section{The group, the lattice and the associated theta functions}
\label{sec:prelim}

Let the complex vector space $V=\CC^3$, with coordinates $(z_1,z_2,z_3)$, be endowed with its standard Hermitian product $\langle-|-\rangle$ and its standard symmetric bilinear form $(-|-)$, and consider the complex root system ${\Phi}$ in it, consisting of the 42 vectors obtained from $(2,0,0)$, $(0,\alpha,\alpha)$ and $(1,1,\bar\alpha)$, where $\alpha=\frac{1+i\sqrt 7}{2}$, by sign changes and permutations of coordinates. For each $\phi\in {\Phi}$, the reflection of order $2$
$$
r_\phi \colon V\longrightarrow V,\quad 
z\longmapsto z-2\frac{\langle\phi|z\rangle}{\langle\phi|\phi\rangle}\phi
$$
sends ${\Phi}$ onto ${\Phi}$. As $r_\phi=r_{-\phi}$, we obtain 21 reflections in this way. We choose
$$
\phi_1=(0,\alpha,-\alpha),\quad \phi_2=(0,0,2),\quad \phi_3=(1,1,\bar\alpha)
$$
as the basic roots and write $r_i=r_{\phi_i}$, $i=1,2,3$:
$$
r_1={\scriptsize \begin{pmatrix}
        1&0&0\\
        0&0&1\\
        0&1&0\end{pmatrix}},\quad
r_2={\scriptsize \begin{pmatrix}
        1&0&0\\
        0&1&0\\
        0&0&{-1}\end{pmatrix}},\quad 
r_3=\tfrac12 {\scriptsize \begin{pmatrix}
        1&{-1}&{-\alpha}\\
        {-1}&1&{-\alpha}\\
        -\bar\alpha&-\bar\alpha&0\end{pmatrix}}.
$$
We define the lattice $\Lambda$ as the root lattice $Q(\Phi)$ of ${\Phi}$, that is, $\Lambda=\sum_{\phi\in\Phi}\ZZ\phi$. We have
$$
\Lambda=\OOO\phi_1+\OOO\phi_2+\OOO\phi_3,\ \ \OOO=\ZZ[\alpha]=\ZZ + \alpha\ZZ.
$$
By~\cite[Section~III.3, pp.~235--236]{Mazur},
$\Lambda$ is the period lattice of Klein's quartic; it can also be represented in the form
$$
\Lambda=\left\{(z_1,z_2,z_3)\in{\OOO}^3 \colon  z_1\equiv z_2 \equiv z_3\ \mathrm{mod}\ \alpha,\ 
z_1+z_2+z_3\equiv 0 \ \mathrm{mod}\ \bar\alpha\right\}.
$$

The group $G$ is the full group of complex-linear automorphisms of $\Lambda$. It contains 21 reflections, all of order~$2$, and is generated by the three basic reflections:
$G=\langle r_1,r_2,r_3\rangle$. According to~\cite[Equation~(10.1)]{Sh-To}, the following relations are defining for $G$:
\begin{equation}\label{rel-ri}
r_1^2=r_2^2=r_3^2=(r_1r_2)^4=(r_2r_3)^4=(r_3r_1)^3=(r_1r_2r_1r_3)^3=1.
\end{equation}
Klein's simple group of order $168$ is the unimodular part of $G$:
$$
H=\{h\in G\colon \det (h)=1\}.
$$
It is generated by the antireflections $\rho_\phi:=-r_\phi$; of course, the antireflections $\rho_i=-r_i$ associated to the three basic roots $\phi_i$ ($i=1,2,3$) suffice to generate $H$. According to~\cite{Po}, there is a unique extension of $\Lambda$ by $G$, the split one, or the semi-direct product $\Gamma=\Lambda\rtimes G$, and it is a complex crystallographic reflection group.

Note that $\Lambda$ contains the sublattice $\alpha Q$, homothetic to the root lattice $Q=Q(C_3)$ of the real root system $C_3$, and thus $G$ contains the Weil group $W=W(C_3)$ of order $48$. The latter consists of all the monomial matrices of size $3$ whose only non-zero elements are $\pm 1$. We have $W=G\cap\mathbf O(3)$, and $G$ is the union of seven cosets $g_7^iW$ ($i=0,\ldots,6$) for some element $g_7$ of order $7$. We can choose $$g_7=\rho_1\rho_2\rho_3=-r_1r_2r_3= \tfrac12\left(\begin{matrix} {-1}&1&\alpha\\ -\bar\alpha&-\bar\alpha&0\\ 1&{-1}&\alpha
\end{matrix}\right).$$ 
This is a handy way to enumerate all the elements of $G$, say, for computer checks of some properties.

Let $M$ denote the weight lattice of $C_3$, $M=Q^*=\ZZ^3+\frac12(1,1,1)$. Then $\Lambda=2\bar\alpha M+\alpha Q$ is generated by the columns of the 6-by-3 period matrix $\Omega=(\omega_1\,|\,\omega_2)$, where $\omega_1,\omega_2$ are square blocks of size 3, the columns of $\omega_2$ being $\alpha$ times the elements of a basis of $C_3$, and we set
\begin{equation}\label{C-matrix}
\omega_2= \alpha C,\quad \omega_1=-2\bar\alpha (\trans{C})^{-1},\quad C=\begin{pmatrix}1&0&0\\
      -1&1&0\\
      0&-1&2\\\end{pmatrix}.
\end{equation}
The normalized period matrix of $\JJJ$ is obtained by multiplying $\Omega$ on the left by $\omega_2^{-1}$:
\begin{equation}\label{Z-period}
\omega_2^{-1}\Omega=(Z\,|\,I),\quad Z=\omega_2^{-1}\omega_1=\tau B,
\end{equation}
where $\tau=-\bar\alpha^{\,2}$ and
\begin{equation}\label{B-matrix}
B=\left(  \trans{C} C\right)^{-1}=\frac14\begin{pmatrix}
     4&4&2\\
     4&8&4\\
     2&4&3\\
     \end{pmatrix}
\end{equation}
is real symmetric and positive definite, so that $Z=\tau B\in\mathfrak H_3$, where $\mathfrak H_r$ denotes the Siegel half-space of complex symmetric matrices of size $r$ with positive-definite imaginary part.

Let $c_1,c_2\in\RR^r$ be two vectors, $k\in\ZZ$, $k\geq 0$, $Z\in \mathfrak H_r$. The classical theta function with period $Z$ and characteristic $c_1Z+c_2$ and of degree $k$ is the complex-valued function
$$
\CC^r\ni v\longmapsto \theta_k\thch{c_1}{c_2}(v,Z)= \sum_{u\in\ZZ^r}
e^{2\pi ik[\trans{(v+c_2)}(u+c_1)+\frac12\trans{(u+c_1)}Z (u+c_1)]}.
$$
When $k=1$, the subscript $k$ is usually omitted.
For $k=1$ and $c_1=c_2=0$, the function $\theta\thch{0}{0}(\bullet,Z)$ represents a section of a uniquely determined line bundle $\LLL$ on the principally polarized abelian variety $A=\CC^r/(Z\ZZ^r+\ZZ^r)$, and then
for any $k,c_1,c_2$, the function $\theta_k\thch{c_1}{c_2}(\bullet,Z)$ represents a section of the line bundle
$T_{c_1Z+c_2}^*(\LLL^k)$, the pullback of the tensor power $\LLL^k$ of $\LLL$ by the translation by the point
$c_1Z+c_2 \mod (Z\ZZ^r+\ZZ^r)$ of $A$. If we choose a set $P_k$ of representatives of $\frac1k\ZZ^r/\ZZ^r$, then the $k^3$ theta functions from $\{\theta_k\thch{c_1}{c_2}(\bullet,Z)\}_{m\in P_k}$ represent a basis of
$H^0(A,T_{c_1Z+c_2}^*(\LLL^k))$. See \textit{e.g.}~\cite{BL}.

\begin{defn}\label{theta-fun}
Let $\Lambda$ be as above, $Z\in \mathfrak H_3$ as in \eqref{Z-period}, $r=3$, $k\in\ZZ$ with $k\geq 1$ and $m\in \frac{1}{k} \ZZ^3$. We define the theta function for $\Lambda$ of degree $k$ with characteristic $m$ by the formula
$$
\theta_{m,k}(z)=\theta_k\thch{m}{0}(\omega_2^{-1}z,Z)\quad \text{for any}\ z\in V.
$$
\end{defn}

There is a unique line bundle $\LLL$ on $\JJJ=V/\Lambda\simeq V/(Z\ZZ^3+\ZZ^3)$ defining a principal polarization such that the functions $\{\theta_{m,k}\}_{m\in P_k}$ form a basis of $H^0(\JJJ,\LLL^k)$, where $P_k$ is a set of representatives of $\frac1k\ZZ^3/\ZZ^3$, for all $k\geq 1$.

We denote by $\Sp(2r,\ZZ)$ the symplectic group of automorphisms of $\ZZ^{2r}$ preserving the skew-symmetric bilinear form given by the matrix $E$:
$$
\Sp(2r,\ZZ)=\left\{ A\in M_{2r}(\ZZ) \colon \trans{A}EA=E\right\}, \quad E=E_r=\begin{pmatrix}0&-I_r\\I_r&0\end{pmatrix},
$$
where $I$ (or $I_r$) denotes the identity matrix (of size $r$).  We represent $E$ and the matrices $\gamma=\begin{psmallmatrix}a&b\\c&d\end{psmallmatrix}\in \Sp(2r,\ZZ)$ by their blocks of size $r$. A crucial ingredient of our computation of the action of $G$ on the theta functions $\theta_{m,k}$ is a transformation formula under modular transformations. The action of the modular transformation $\gamma$ on a classical theta function is defined by
\begin{equation}\label{mod-transf-def}
\left(\theta_k\thch{c_1}{c_2}\right)^\gamma(v,Z)=
\theta_k\thch{c_1}{c_2}\left(\left(\trans{(cZ+d)}\right)^{-1}v,(aZ+b)(cZ+d)^{-1}\right).
\end{equation}
The following is a particular case of Igusa's theorem for a polarization of type $(k,k,\ldots,k)$. 

\begin{theorem}[\textit{cf.} \protect{\cite[Theorem II.5.6]{Ig}}]\label{Igusa}
For every $(v,Z)\in \CC^r\times\mathfrak H_r$, $c_1,c_2\in\RR^g$, $m\in P_k$ and
$\gamma=\begin{psmallmatrix}a&b\\c&d\end{psmallmatrix}\in \Sp(2r,\ZZ)$, the matrix $cZ+d$ is invertible, and we have
$$
\left(\theta_k\thch{m+c'_1}{c'_2}\right)^\gamma(v,Z)=e^{\pi i k \trans{v}{(cZ+d)}^{-1}cv}\det(cZ+d)^{\frac12}\cdot
\sum_{m'\in P_k} u_{m,m'}\theta_k\thch{m'+c_1}{c_2}(v,Z),
$$
where
$$
\begin{bmatrix} c_1'\\c_2'\end{bmatrix} =
\begin{pmatrix}a&b\\c&d\end{pmatrix}\begin{bmatrix}c_1\\c_2\end{bmatrix}+
\frac12\begin{bmatrix} (c\trans{d})_0\\ (a\trans{b})_0\end{bmatrix},
$$
$(u_{m,m'})\in\mathbf U(k^r)$ and $(h)_0$ denotes the column vector of the diagonal elements of $h$ for any square matrix $h$.
\end{theorem}

In the next section, we will represent the automorphisms from $G$ by modular transformations and compute explicitly the matrices $(u_{m,m'})$ for even $k$. As we will see, the presence of the half-integer inhomogeneous term in the transformation formula for the characteristics $c_1,c_2$ implies the non-invariance of $\LLL$ under the action of $G$; however, the even powers of $\LLL$ are $G$-invariant.

\section{Theta transformation formula}
\label{sec:theta formula}

We keep the notation of the previous section. The elements of $G$ are complex 3-by-3 matrices leaving invariant the lattice $\Lambda$. As $\Lambda$ is generated by the columns of the 6-by-3 matrix
$\Omega=(\omega_1\, | \, \omega_2)$,  we can associate
to each $g\in G$ a matrix 
$\gamma=\gamma_g=\left(\begin{smallmatrix}a&b\\c&d\end{smallmatrix}\right)\in \Sp(6,\ZZ)$ in such a way that
$$
(\omega_1\, | \, \omega_2)\begin{pmatrix}\trans{a}&\trans{c}\\ \trans{b}&
\trans{d}\end{pmatrix}
=g (\omega_1\, | \, \omega_2).
$$
Obviously, the map $g\mapsto\gamma_g$ is a group homomorphism.

\begin{lemma} \label{miracle} Let $g\in G$, and let $\gamma=\left(\begin{smallmatrix}a&b\\c&d\end{smallmatrix}\right)$ be as above. Then the following properties hold:
\begin{enumerate}
\item\label{miracle-1}  
$(aZ+b)(cZ+d)^{-1}=Z$;
\item\label{miracle-2}
$\det (cZ+d)=\det g=\pm 1$;
\item\label{miracle-3}
$\det d = \pm 1$.
\end{enumerate}
\end{lemma}

\begin{proof}
\eqref{miracle-1}~ We have
\begin{multline*}
\hspace*{-1em}\trans{ Z} =Z=\omega_2^{-1}\omega_1=(g\omega_2)^{-1}(g\omega_1)=
 (\omega_1\trans{ c}+\omega_2\trans{ d})^{-1}\cdot  (\omega_1\trans{a}
 +\omega_2\trans{b})=\\
  (Z\trans{ c}+\trans{ d})^{-1}(Z\trans{ a}+\trans{ b})= 
  \trans{\left((aZ+b)(cZ+d)^{-1}\right)}.
\end{multline*}

\eqref{miracle-2}~ By the definition of $\gamma$, $Z\trans{ c}+\trans{ d}=\trans{ (cZ+d)}$ is the matrix of $g$ in the basis of $\CC^3$ given by the columns of $\omega_2$, so $\det g
=\det \trans{ (cZ+d)}$.

\eqref{miracle-3}~ This is verified by a direct computation of the matrices $\gamma_g$  for all elements $g\in G$.
\end{proof}

\begin{corollary}\label{TF-long-ch}
For $z\in V$, we have
$$
g\cdot\theta_{0,1}(z)=\left(\theta\thch{0}{0}\right)^{\gamma_g}(v,Z)=
\chi_{g}\theta\thch{\nu'}{\nu^{\prime\prime}}(v,Z),
$$
where 
$$v=\omega_2^{-1}z, \quad
\chi_{g}=e^{\pi i \trans{v}(cZ+d)^{-1}cv}(\det g)^{1/2},
\quad
\begin{bmatrix}{\nu'}\\{\nu^{\prime\prime}}\end{bmatrix}
= -\frac12\begin{pmatrix}\trans d&-\trans b\\-\trans c&\trans a\end{pmatrix}
\begin{bmatrix}{({ c}\trans d)_0}\\{({ a}\trans b)_0}\end{bmatrix}.
$$
\end{corollary}

\begin{proof}
This immediately follows from Lemma~\ref{miracle}, Igusa's theorem with $k=1$ and the inversion formula
$$\begin{pmatrix}{ a}&{ b}\\ { c}&{ d}\end{pmatrix}^{-1}
= \begin{pmatrix}\trans d&-\trans b\\-\trans c&\trans a\end{pmatrix}$$
for matrices $\begin{psmallmatrix}a&b\\c&d\end{psmallmatrix}\in  \Sp(2r,\ZZ)$.
\end{proof}

We see that the theta function $\theta_{0,1}$, representing a section of $\LLL$, acquires a half-integer characteristic upon the action by an element $g\in G$ whenever the diagonal elements of the integer matrices $c\trans{ d}$ and $a\trans{ b}$ are not all even. In this case, $g^*\LLL$ is not $\LLL$ but the translation of $\LLL$ by a point of order $2$.

\begin{corollary}\label{non-inv}\leavevmode
\begin{enumerate}
\item For $g\in G$, we have the equivalence $g^*\LLL\simeq \LLL\Leftrightarrow g\in W$, where $W=W(C_3)=G\cap \mathbf O(3)$ is the subgroup of real matrices in $G$.
\item $g^*\LLL^k\simeq \LLL^k$ for all $g\in G \Leftrightarrow k$ is even.
\end{enumerate}
\end{corollary}

\begin{proof}
By a direct computation, we verify that for $g\in G$, the diagonal elements of the integer matrices $c\trans{ d}$ and $a\trans{ b}$ are all even if and only if $g\in W$. This implies both assertions.
\end{proof}

Thus the problem of calculating the action of $G$ on $H^0(\JJJ,\LLL^k)$ has sense only for even $k$. We are now going to calculate, in a particular case and only for even $k$, the matrix $(u_{m,m'})$ from Igusa's theorem up to proportionality; we denote the matrix we find by $(\tilde u_{m,m'})$, as it is a multiple of Igusa's matrix $(u_{m,m'})$ which is not necessarily unitary.

For $\left(\begin{smallmatrix}a&b\\c&d\end{smallmatrix}\right)\in \Sp(2r,\ZZ)$ with $\det d=\pm 1$, we define
$$\tilde a=a-bd^{-1}c=(\trans{ d})^{-1}, \quad 
\tilde b=\trans{\tilde b}=b\trans{ \tilde a}, \quad 
\tilde c=\trans{\tilde c}=-c\trans{d}.$$

\begin{theorem}\label{TTF}
Let $k$ be a positive even integer, $P_k$ a set of representatives of $\frac1k\ZZ^r/\ZZ^r$, $\gamma\in \Sp(2r,\ZZ)$ such that $\det d=\pm 1$ and $Z=\tau B$ for a real symmetric positive-definite matrix $B$ of size $r$, where $\tau \in\CC$, $\Im \tau>0$.  Then ${\theta_k\thch{m}{0}}^\gamma= \chi\sum_{m'\in P_k} \tilde u_{m,m'}\theta_k\thch{m'}{0}$, where $\chi=\chi_\gamma(v,Z)$ is a nowhere-vanishing analytic function on $\CC^r\times\mathfrak H_r$, depending on $\gamma$, and
$$
\tilde u_{m,m'}=e^{\pi i k\tilde b[m]}\sum_{\hat m\in P_k}
e^{2\pi i k 
\trans{\left(m-d m'+\frac12\tilde c\hat m\right)}\hat m}.
$$
\end{theorem}

\begin{proof}
We decompose $\gamma$ in a product of elementary transformations as follows:
$$
\gamma=
\raisebox{-1.6ex}{$\substack{\begin{pmatrix}1&\tilde b\\ 0&1\end{pmatrix}\\ \, \verteq\\ \sigma_1}$}
\raisebox{-1.6ex}{$\substack{\begin{pmatrix}0&1\\ -1&0\end{pmatrix}\\ \, \verteq\\ \sigma_2}$}
\raisebox{-1.6ex}{$\substack{\begin{pmatrix}1&\tilde c\\ 0&1\end{pmatrix}\\ \, \verteq\\ \sigma_3}$}
\raisebox{-1.6ex}{$\substack{\begin{pmatrix}0&-1\\ 1&0\end{pmatrix}\\ \, \verteq\\ \sigma_4}$}
\raisebox{-1.6ex}{$\substack{\begin{pmatrix}\tilde a&0\\ 0&d\end{pmatrix}\\ \, \verteq\\ \sigma_5}$}
$$
and apply the factors of this decomposition successively.

{\em Step}~1.~ 
$\left(\theta_k\thch{m}{0}\right)^{\sigma_1}(v,Z)=
\theta_k\thch{m}{0}(v,Z+\tilde b)=$
$$
\sum_{u\in\ZZ^r}e^{2\pi ik(\trans{ v}(u+m)+\frac12\trans{(u+m)}Z(u+m))}\cdot
e^{\pi i k \tilde b[u]}\cdot e^{2\pi i k \trans{ m} \tilde bu}\cdot e^{\pi i k \tilde b[m]} =
e^{\pi i k \tilde b[m]} \theta_k\thch{m}{0}(v,Z),
$$
where we use the notation $A[u]=\trans{ u} A u$ for any symmetric matrix $A$ of size $r$ and any vector
$u\in\CC^r$, and $e^{\pi i k \tilde b[u]}= e^{2\pi i k \trans{ m} \tilde bu}=1$ for any $u\in\ZZ^r$, $k$ even.

{\em Step}~2.~ We have
$$
\left(\theta_k\thch{m}{0}\right)^{\sigma_1\sigma_2}(v,Z)=
e^{\pi i k \tilde b[m]} \theta_k\thch{m}{0}(-Z^{-1}v,-Z^{-1})=
e^{\pi i k \tilde b[m]} \theta_k\thch{-m}{0}(Z^{-1}v,-Z^{-1}),
$$
$$
\theta_k\thch{-m}{0}(Z^{-1}v,-Z^{-1})=
\sum_{u\in\ZZ^r}e^{\frac{2\pi i k}{\tau}
(\trans{ v} B^{-1}(u-m)-\frac12\trans{(u-m)}B^{-1}(u-m))}=
e^{\frac{\pi i k}{\tau}A[v]}
\sum_{u\in\ZZ^r}e^{-\frac{\pi i k}{\tau}A[u-m-v]},
$$
where $A=B^{-1}$ and we used the hypothesis that $Z=\tau B$.
To transform the latter expression, we apply the Jacobi inversion formula (see for example~\cite[Chapter~VI]{Gu}):
$$
\sum_{u\in\ZZ^r}e^{\pi i t  A[x+u]}=\frac1{\sqrt{(-it)^r\det A}}\sum_{u\in\ZZ^r}
e^{-\frac{\pi i}tA^{-1}[u]+2\pi i\trans{ x}u} 
$$
for $x\in \RR^r, \ t\in\CC, \ \Im t >0,\ A\in M_r(\RR),\ \trans{ A}=A,\ A>0$.
We set $t=-k/\tau$ and $x=-m-v$, and we obtain
$$
\left(\theta_k\thch{m}{0}\right)^{\sigma_2}(v,Z)=
\chi_0
\sum_{u\in\ZZ^r}
e^{\pi i k Z[\frac{u}{k}]-2\pi i k\trans{ (m+v)}\frac{u}{k}}
= \chi_0 \sum_{u'\in\frac1k\ZZ^r}e^{\pi i k Z[u']-2\pi i k\trans{ (m+v)}u'},
$$
where $\chi_0(v,Z)=\frac{e^{\pi i k Z^{-1}[v]}}{\sqrt{(-ik)^r\det Z^{-1}}}$.
Now each $u'\in \frac1k\ZZ^r$ has a unique representation $u'=u+m'$ with $u\in\ZZ^r$ and $m'\in P_k$, and we obtain
$$
\left(\theta_k\thch{m}{0}\right)^{\sigma_1\sigma_2}(v,Z)=\chi_0e^{\pi i k \tilde b[m]} 
\sum_{m'\in P_k} e^{2\pi i k \trans{m}m'}\theta_k\thch{m'}{0}(v,Z),
$$
where $\chi_0=\chi_0(v,Z)$ is the nowhere-vanishing analytic function in $v,Z$ defined above.

{\em Step}~3.~ We have
$$
\left(\theta_k\thch{m}{0}\right)^{\sigma_1\sigma_2\sigma_3}(v,Z)=
\chi_1 e^{\pi i k \tilde b[m]} 
\sum_{m'\in P_k} e^{2\pi i k (\trans{ m}m'+\frac12\tilde c[m'])}\theta_k\thch{m'}{0}(v,Z),
$$
where $\chi_1(v,Z)=\chi_0(v,Z+\tilde c)$.

{\em Step}~4.~ The calculation is similar to that in Step 2:
\begin{align*}
\left(\theta_k\thch{m}{0}\right)^{\sigma_1\sigma_2\sigma_3\sigma_4}(v,Z)
&=\chi_1(Z^{-1}v,-Z^{-1})e^{\pi ik\tilde b[m]}
\sum_{m'\in P_k}e^{2\pi ik \trans{ (m+\frac12\tilde c m')}m'}\theta_k\thch{m}{0}(Z^{-1}v,-Z^{-1})\\
&=
\chi_2\sum_{m'\in P_k}\sum_{m^{\prime\prime}\in P_k}\zeta_{m',m^{\prime\prime}}
\theta_k\thch{m^{\prime\prime}}{0}(v,Z),
\end{align*}
where $\zeta_{m',m^{\prime\prime}}=e^{\pi ik\tilde b[m]}e^{2\pi ik \trans{ (m+\frac12\tilde c m'-m^{\prime\prime})}m'}
$ 
and $\chi_2(v,Z)=\chi_1(Z^{-1}v,-Z^{-1})\chi_0(v,Z)$.

{\em Step}~5.~
$\theta_k\thch{m}{0}^\gamma(v,Z)=\chi{\dss \sum_{m',m^{\prime\prime}}}
\zeta_{m',m^{\prime\prime}}
\theta_k\thch{m^{\prime\prime}}{0}^{\left(\begin{smallmatrix}\tilde a&0\\ 0&d\end{smallmatrix}\right)}(v,Z)=
\chi{\dss \sum_{m',m^{\prime\prime}}}
\zeta_{m',m^{\prime\prime}}
\theta_k\thch{\trans{\tilde a} m^{\prime\prime}}{0}(v,Z)$
$$
= \chi{\dss \sum_{m',m^{\prime\prime}}}
\zeta_{m',(\trans{\tilde a})^{-1}m^{\prime\prime}}
\theta_k\thch{ m^{\prime\prime}}{0}(v,Z)=
\chi{\dss \sum_{m^{\prime\prime}}}
\tilde  u_{m,m^{\prime\prime}}
\theta_k\thch{ m^{\prime\prime}}{0}(v,Z),\hphantom{=\;\,}
$$
where $\tilde  u_{m,m^{\prime\prime}}$ is as in the statement of the theorem and
$\chi(v,Z)=\chi_2((\trans{ d})^{-1}v,(\trans{ d})^{-1}Zd^{-1})$.
\end{proof}

Together with Igusa's theorem and Lemma~\ref{miracle}, Theorem~\ref{TTF} obviously implies the following corollary.

\begin{corollary}\label{G-to-PU}
For every even $k\geq 0$, the application of Theorem~\ref{TTF} to the theta functions $\theta_{m,k}$ for the lattice $\Lambda$ introduced in Definition~\ref{theta-fun} provides a map $g\mapsto \tilde U_g$, where $\tilde U_g$ is the complex matrix $(\tilde u_{m,m^{\prime}})$ of size $k^3$ defined in the statement of Theorem~\ref{TTF} with $\gamma_g$ in place of $\gamma$, and this map provides a group homomorphism from $G$ to the projective unitary group $\mathbf P\mathbf U(k^3)=\mathbf U(k^3)/\text{$($homotheties$)$}$.
\end{corollary}

\section{Unitary action of \texorpdfstring{$\boldsymbol{G}$}{G} on theta functions and its character}
\label{sec:Unitary action}

Neither Igusa's theorem nor our approach applied in Theorem~\ref{TTF} allows us to conclude that the matrices $\tilde U_g$ can be normalized by multiplying by some constants $\epsilon_g$ in such a way that the normalized map $g\mapsto U_g=\epsilon_g\tilde U_g$ is a group homomorphism $G\to \mathbf U(k^3)$. However, we managed to find convenient constants $\epsilon_g$ by trial and error. We will write $U_g^{(k)}, \tilde U_g^{(k)}$ when we want to specify the degree $k$ of theta functions on which $U_g$, $\tilde U_g$ act.

\begin{proposition}\label{U-matrices}
Let $r_1,r_2,r_3$ be the basic reflections generating $G$ introduced in Section~\ref{sec:prelim}. Set $U_j=\frac1{k^3}\tilde U_{r_j}$ for $j=1,2$ and $U_3=\frac1{ik^3}\tilde U_{r_3}$. Then $U_1,U_2,U_3\in \mathbf{U}(k^3)$. We denote by $U_G^{(k)}$ the subgroup of $\mathbf{U}(k^3)$ generated by these three matrices.
\begin{enumerate}
\item 
The $U_j$ satisfy the same relations \eqref{rel-ri} as the basic reflections $r_j$:
$$U_1^2=U_2^2=U_3^2=(U_1U_2)^4=(U_2U_3)^4=\\ (U_3U_1)^3=(U_1U_2U_1U_3)^3=1.$$
\item 
$U_G^{(2)}\simeq H$ and  $U_G^{(k)}\simeq G$ for all even $k\geq 4$.
\end{enumerate}
\end{proposition}

\begin{proof} We verify this for $k=2,4,6$ by direct computation using the computer algebra system Macaulay2~\cite{M2}; the result for all even $k$ follows from the fact that  the algebra of even-degree theta functions on a  p.p.a.v. is generated in degrees at most~$6$.
\end{proof}

We thus have a unitary representation $\boldsymbol\rho_k$ of $G$ on the space $H^0(\JJJ,\LLL^k)$ of dimension $k^3$ for each even $k\geq 0$, defined by substituting the $U_j$ for the $r_j$ in the words in the $r_j$ defining all the elements of $G$. We denote by $\boldsymbol\chi_k$ the character of this representation. In order to determine it, we start by fixing the choice of representatives of the conjugacy classes of $G$.  Klein's simple group $H$ has six conjugacy classes, represented by the following elements:
$$
g_1=1,\quad g_2=\rho_1,\quad g_3=\rho_1\rho_3\rho_1\rho_2,
\quad  g_4=\rho_1\rho_2,\quad g_7=\rho_1\rho_2\rho_3,\quad g_7^{-1},
$$
where the $\rho_i=-r_i$ are the basic antireflections and the subscript $p$ in $g_p$ stands  for the order of $g_p$.

The conjugacy classes of $G$ are deduced from these in an obvious way: to every conjugacy class $\Cl_H(g)$ in $H$ correspond two conjugacy classes in $G$ of the same length: $\Cl_G(g)=\Cl_H(g)$ and $\Cl_G(-g)=-\Cl_H(g)$. Also, to each irreducible representation $f$ of $H$ correspond two irreducible representations of $G$, $\tilde f=f\circ\pi$ and $\tilde f\otimes \det$, where $\pi\colon G\to H\simeq G/\langle -1\rangle$ is the natural surjection. 

The lengths of the conjugacy classes of $H$ are given by the following table, providing the characters of $H$. The characters of $G$ are easily deduced from it.

\begin{center}
\begin{tabular}{|c|c|c|c|c|c|c|}
\hline\rule[0pt]{0pt}{\heightof{$g_y^{-1}$}+.4ex}
\ \ \ \ $g$\ \ \ \ & $g_1$ & $g_2$ & $g_3$ & $g_4$ & $g_7$ & $g_7^{-1}$ \\[.2ex]
\hline
\ \ \ \ $|\mathrm{Cl}_{H}(g)|$\ \ \ \  &1&21&56&42&24&24\ \\
\hline
\ \ \ \  \ $\chi_{\mathbf 1}$\ \  \ \ \  \ & 1&1&1&1&1&1\\
\hline
\ \ \ \  \ $\chi_{\mathbf 3}$\ \  \ \ \  \ & 3&$-1$&0&1&$-\bar\alpha$&$-\alpha$\\
\hline
\ \ \ \  \ $\chi_{\bar{\boldsymbol 3}}$\ \  \ \ \  \ & 3&$-1$&0&1&$-\alpha$&$-\bar\alpha$\\
\hline
\ \ \ \  \ $\chi_{\mathbf 6}$\ \  \ \ \  \ & 6&2&0&0&$-1$&$-1$\\
\hline
\ \ \ \  \ $\chi_{\mathbf 7}$\ \  \ \ \  \ & 7&$-1$&1&$-1$&0&0\\
\hline
\ \ \ \  \ $\chi_{\mathbf 8}$\ \  \ \ \  \ & 8&0&$-1$&0&1&1\\
\hline
\end{tabular}
\end{center}
\bigskip

\begin{theorem}\label{char-rho}
For any even $k> 0$, the character $\boldsymbol{\chi}_k$ of $\boldsymbol{\rho}_k$ takes the following values on the above representatives of the conjugacy classes:

\begin{center}
\begin{tabular}{|c|c|c|c|c|c|c|c|c|c|c|c|}
\hline
$g_1$ & $-g_1$ & $g_2$ & $-g_2$ & $g_3$ & $-g_3$ & $g_4$ &$-g_4$&  $g_7$ & $-g_7$ & \T$g_7^{-1}$& \T $- g_7^{-1}$  \\[.3ex]
\hline
$k^3$  & $8$ & $2k$ & $k^2$ & $k$ & $2$ & $k$ & $3+(-1)^{\frac k2}$\T&
\makebox[3em]{${\genfrac{(}{)}{1pt}{1}{k}{7}\ {\rm if}\ 7\nmid k,\T\atop
 {-i\sqrt7\ {\rm if}\ 7\mid k\B}}$} & $1$ & \makebox[3em]{${\genfrac{(}{)}{1pt}{1}{k}{7}\ {\rm if}\ 7\nmid k,\T\atop
 {i\sqrt7\ {\rm if}\ 7\mid k\B}}$} & $1$   \\
\hline
\end{tabular}\\[10pt]
\end{center}

where
$$
\genfrac{(}{)}{1pt}{0}{k}{7}=\left\{\begin{array}{lll}
1 & \mbox{\rm if} & k\equiv 1,2\ \mbox{\rm or}\ 4 \mod 7,\\
-1 & \mbox{\rm if} & k\equiv 3,5\ \mbox{\rm or}\ 6 \mod 7,\\
0 & \mbox{\rm if} & k\equiv 0 \mod 7
\end{array}\right. 
$$
is the Legendre symbol.

\end{theorem}

\begin{proof}
  The result is obvious for $g_1$ and needs some reasoning, following the same pattern, in the other cases. We will illustrate this reasoning on the example of $g=g_7$, where the details of the calculation are the  most involved. We first observe that the normalization constant for $g=g_7$ is $\epsilon_{g}=\frac1{ik^3}$, so that
$$
\boldsymbol{\chi}_k(g)=\tr U_g=\frac{1}{ik^3}\Sigma_k,\quad \text{where}\ \Sigma_k=\sum_{m\in P_k}\tilde u_{m,m}.
$$
By Theorem~\ref{TTF},
$$
\Sigma_k=
\sum_{m, m'\in P_k}e^{\pi i k  
\left(2\trans{ m'}(\id-d) m+\tilde b[m]+\tilde c[m']\right)}=
\sum\limits_{0\leq x_i \leq k-1}e^{\frac{\pi i}{k}K[x]},
$$
where we pass to the summation over the integer column vector 
$x=\trans{(x_1,\ldots,x_6)}=k\trans{ (m',m)}\in\ZZ^6$, $K[x]=\trans{ x} K x$ as before and $K$ is the following integer matrix of size~6:
$$
K=\begin{pmatrix}\tilde c&I_3-d\\ I_3-\trans{d}&\tilde b\end{pmatrix}.
$$
Explicitly, we have
$$
\gamma_g=\begin{pmatrix}
      \begin{smallmatrix}
      -1&0&1\TopStrut{2} \\ 
      -1&1&0 \\ 
      0&0&1
      \BStrut{0.8} 
      \end{smallmatrix} 
      & \rvline &
      \begin{smallmatrix}
      0&-2&-1\TopStrut{2} \\ 
      0&-4&-2\\ 
      -1&-3&-2
      \BStrut{0.8} 
      \end{smallmatrix} 
      \\ \hline
      \begin{smallmatrix}
      0&0&0\TopStrut{1.5}\\ 
      -1&1&-1\\ 
      1&-1&2
      \BStrut{1.5} 
      \end{smallmatrix} 
      & \rvline &
      \begin{smallmatrix}
      -1&-1&0\TopStrut{1.5}\\ 
      1&0&0\\ 
      -1&-1&-1 
      \BStrut{1.5}
      \end{smallmatrix}
     \end{pmatrix}, \quad 
K=\begin{pmatrix}
\begin{smallmatrix}
      0&0&0\TopStrut{2}\\
      0&1&-1\\
      0&-1&2
      \BStrut{0.8} 
      \end{smallmatrix} 
      & \rvline &
      \begin{smallmatrix}
      2&1&0\TopStrut{2}\\
      -1&1&0\\
      1&1&2
      \BStrut{0.8} 
      \end{smallmatrix} 
      \\ \hline
      \begin{smallmatrix}2&-1&1\TopStrut{1.5}\\
      1&1&1\\
      0&0&2\BStrut{1.5} \end{smallmatrix} & \rvline &
      \begin{smallmatrix}\:1\ &2\ &1\TopStrut{1.5}\\
      2&4&2\\
      1&2&2
       \BStrut{1.5}\end{smallmatrix}
     \end{pmatrix}.
$$
By Gauss diagonalization over $\ZZ$, we reduce the quadratic form $x\mapsto K[x]$ to the diagonal representation
$y=(y_1,\ldots,y_6)\mapsto y_1^{ 2}+y_2^{ 2}+y_3^{ 2}-y_4^{ 2}-y_5^{ 2}-7y_6^{ 2}$. Thus
\begin{align*}
\Sigma_k&=\sum_{0\leq x_i \leq k-1}e^{\frac{\pi i}k K[x]}=
\sum_{0\leq y_i \leq k-1}e^{\frac{\pi i}k(y_1^{ 2}+y_2^{ 2}+y_3^{ 2}-y_4^{ 2}-y_5^{ 2}-7y_6^{ 2})}\\
&=\left(\sum_{y=0}^{k-1}e^{\frac{\pi i}ky^2}\right)^3\left(\sum_{y=0}^{k-1}e^{-\frac{\pi i}ky^2}\right)^2
\left(\sum_{y=0}^{k-1}e^{-\frac{7\pi i}ky^2}\right).
\end{align*}
The exponential sums in the last line belong to the class of Gauss sums. We use the following formula for Gauss sums from~\cite{BEW}:
$$
G(q,r)=\sum_{n=0}^{r-1}e^{\frac{2\pi iq}rn^2}=
(1+i)\kappa_q^{-1}\sqrt r\genfrac{(}{)}{1pt}{0}{r}{q}
$$
if $2\nmid q,\ 4| r$ 
and gcd$(q,r)=1$, where
$\kappa_q=\left\{\begin{array}{lll}
1&\text{if}&q\equiv 1\mod 4,\\
i&\text{if}&q\equiv 3\mod 4
\end{array}\right.
$
and \ $\genfrac{(}{)}{1pt}{0}{r}{q}$ \ is the Jacobi symbol (which coincides with the Legendre symbol when $q\TopStrut{3}$ is prime). Applying this formula, we obtain 
\begin{gather*}
\sum_{y=0}^{k-1}e^{\frac{\pi i}ky^2}=\frac12 G(1,2k)=\frac{1+i}{\sqrt2}\sqrt k,
\\
\sum_{y=0}^{k-1}e^{-\frac{7\pi i}ky^2}=
\begin{cases}
\displaystyle{\frac12 \bar{G(7,2k)}=\frac{1+i}{\sqrt2}\genfrac{(}{)}{1pt}{0}{k}{7}\sqrt k}
&\text{if}\  7\nmid k,\\[3ex]
\displaystyle{\frac72 \,\bar{G(1,2k_1)}=\frac{1-i}{\sqrt2}\sqrt{7k}}
& \text{if}\ k=7k_1.
\end{cases}
\end{gather*}
We thus finish the calculation of $\Sigma_k$ and obtain the value of $\boldsymbol{\chi}_k(g)$ given in the table.
\end{proof}

By taking the scalar product of $\boldsymbol{\chi}_k$ with the trivial character, we obtain the following. 

\begin{theorem}\label{main}
The Hilbert function of the algebra of invariant theta functions
$$
S\left(\LLL^2\right)^G=\bigoplus_{p=0}^\infty H^0\left(\JJJ,\LLL^{2p}\right)^G
$$
is given by
$$
h_{S(\LLL^2)^G}(\tfrac k2) =\tfrac1{336}\left[k^{3}+21\,k^{2}+140\,k+294 + (-1)^{\frac k2}
\times 42+48\genfrac{(}{)}{1pt}{1}{k}{7}\right].
$$
This coincides with the Hilbert function of the second Veronese algebra of\, $\PP(1,2,4,7)$.
\end{theorem}

\begin{proof}
The formula for the Hilbert function is an immediate consequence of Theorem~\ref{char-rho}. The Hilbert series for 
$\PP=\PP(1,2,4,7)$ is given by
$$
\HSer_\PP(t)=\frac1{(1-t)(1-t^2)(1-t^4)(1-t^7)},
$$
and that of the second Veronese of $\PP$ is
$$
\HSer^{(2)}_\PP(t)=\frac12\left(\HSer_\PP\left(t^{1/2}\right)+\HSer_\PP\left(-t^{1/2}\right)\right)=\frac{1+t^4}{\left(1-t\right)^{2}\left(1-t^2\right)\left(1-t^7\right)}.
$$
Denote by $h_p$ the coefficients of the latter series, so that $\HSer^{(2)}_\PP(t)=\sum_{p\geq 0}h_pt^p$; by~\cite[Theorem 4.4.1]{Sta}, the Hilbert function
$p\mapsto h_p$ is a rational quasi-polynomial of degree $3$ whose coefficients are periodic with period 14. The dominant coefficient being constant, it suffices to compare the initial segments of the sequences $h_{S(\LLL^2)^G}(p)$ and $h_p$ of length 42 to see that they coincide for all $p\geq 0$. In this way, we conclude the proof.
\end{proof} 

\section{Hypersurface in a weighted projective space of dimension 4}
\label{sec:Main Theorem}

For any theta function $\theta$ of even degree $k>0$, we denote by $\pazocal R_G^{(k)}$ the Reynolds operator
\begin{equation}\label{Reynolds}
\pazocal R_G^{(k)}(\theta)=
\frac1{336}\sum_{g\in G}U_g^{(k)}(\theta).
\end{equation}
We will call $\pazocal R_G^{(k)}(\theta)$ the $G$-average of $\theta$; for given $k$, the $G$-averages of theta functions of degree $k$ represent sections of $\LLL^k$ which generate $H^0(\JJJ,\LLL^k)^G$ over $\CC$.

We will identify sections of $\LLL^k$ with regular functions on the total space $\mathbf V(\LLL^{-1})$ of the line bundle $\mathcal L^{-1}$ which are fiber-homogeneous of degree $k$. Explicitly, $\mathbf V(\LLL^{-1})$ is the quotient of the trivial line bundle $\CC\times V$ over $V=\CC^3$ by the action of $\Lambda\simeq \{Zu'+u''\ |\ u',u''\in\ZZ^3\}$:
$$
Zu'+u'' \colon (t,v)\longmapsto \left(e^{2\pi i\left(\trans{v}u'+\frac12Z[u']\right)}t,\ v+Zu'+u''\right),
$$
and to each degree $k$ theta function $\theta$, we associate the $\Lambda$-invariant function
\begin{equation}\label{tilde}
\tilde\theta \colon\CC\times V\to\CC,\quad (t,v)\longmapsto t^k\theta(v),
\end{equation}
which descends to a fiber-homogeneous function of degree $k$ on $\mathbf V(\LLL^{-1})$ denoted by the same symbol $\tilde\theta$.

\begin{theorem}\label{main2}
There exist five $G$-invariant theta functions $\varphi_i=\varphi_{i,k_i}$ of degrees $k_i$, where
$(k_0,\ldots,k_4)=(2,2,4,8,14)$, such that the associated sections of the powers $\LLL^{k_i}$ of $\LLL$
generate $S(\LLL^2)^G$. The ideal of relations between these generators is generated by a single relation of weighted degree $16$, so that the quotient $X=\JJJ/G$ is isomorphic to a hypersurface in the weighted projective space $\PP(1,1,2,4,7)$ of weighted degree $8$.
\end{theorem}

\begin{proof} It suffices to show that we can find $G$-invariant theta functions
$\varphi_0,\varphi_1,\varphi_2,\varphi_3$ of respective degrees $2,2,4,8$ such that the associated  fiber-homogeneous functions $\tilde\varphi_i$ on $\mathbf V(\LLL^{-1})$ are algebraically independent. This is done below in Lemma~\ref{Jac-non-0}. Then the functions
$\varphi_i$, $i=0,\ldots,3$, generate the free polynomial subalgebra $\SSS'$ in $\SSS= S(\LLL^2)^G$, and comparing the initial segments of their Hilbert functions,
$$
(h_{\SSS}(p))_{p\geq 0} = (1, 2, 4, 6, 10, 14, 20, 27, 36, 46, 58,\ldots),
$$
$$
(h_{\SSS'}(p))_{p\geq 0} = (1, 2, 4, 6, 10, 14, 20, 26, 35, 44, 56,\ldots),
$$
we see that the four $\varphi_i$ generate $\SSS$ in degrees less than $7$ and $h_{\SSS}(7)=h_{\SSS'}(7)+1$, so that one extra generator of degree $7$, say $\varphi_4$, is needed to generate $\SSS$ up to degree $7$. Furthermore, there 
are 35 monomials in the $\varphi_i$, $i=0,\ldots,3$, of degree $8$, and two more monomials $\varphi_0\varphi_4$,
$\varphi_1\varphi_4$ involving $\varphi_4$, and as $h_{\SSS}(7)=36$ is one less than the number of monomials of degree $8$, there is precisely one linear relation $F_8=0$ between the 37 monomials of degree $8$. 
One easily verifies that $h_{\SSS}(p)=h_R(p)-h_R(p-8)$ for all $p\in \ZZ$, where $R$ denotes a free polynomial algebra with generators of degrees $1,1,2,4,7$, so that the natural map $R/F_8R\to\SSS$ is an isomorphism.
Thus $X$ is isomorphic to a hypersurface $F_8=0$ of degree $8$ in the weighted projective space $\PP(1,1,2,4,7)=\Proj R$.
\end{proof}

For the algebraic independence of $\tilde\varphi_0,\ldots,\tilde\varphi_3$, it is necessary and sufficient that the Jacobian $J=J(\tilde\varphi_i)$ is not identically zero. We will present explicitly an \textit{ad hoc} example of $\varphi_i(v)$ lying in the images of the Reynolds operators $\pazocal R_G^{(k_i)}$ and satisfying this property; some motivation for the choice of the example is given in Remark~\ref{conj-gen}.  We will verify that $J(t_0,v_0)\neq 0$ at a specific point $(t_0,v_0)$ by computing the differentials $\ud\tilde\varphi_i(t_0,v_0)$ approximately as partial sums of their Fourier series, which converge very rapidly. We recall that the period matrix of $\JJJ$ is $Z=\tau B$, where $B$ is given by \eqref{B-matrix} and $\tau=\frac{3+i\sqrt7}{2}$, and we set $q=e^{2\pi i\tau}=-e^{-\pi\sqrt7}$; we also adopt the convention that $q^r=e^{2\pi i\tau r}$ for any $r\in\QQ$.

\begin{lemma}\label{Jac-non-0}
Let the vectors $\xi_i$ ($i=0,\ldots,3$) of $\ZZ^3$ be defined by
$$
\xi_0=0,\quad \xi_1=\frac12(0,0,1),\quad \xi_2=\frac14(1,1,0),\quad \xi_3=\frac18(2,1,1),
$$
the first three being equal to the vectors $\mu_i$ in \eqref{vertices} in  
Section~\ref{sec:various}, with the basis $(b_i)$ of $M$ used to identify $M$ with $\ZZ^3$, and $\xi_3=\frac78\mu_3$.
Define four $G$-invariant theta functions on $\JJJ$ by
$$
\varphi_0=\pazocal R_G^{(2)}(\theta_{\xi_0,2}),\quad \varphi_1=\pazocal R_G^{(2)}(\theta_{\xi_1,2}),\quad \varphi_2=\pazocal R_G^{(4)}(\theta_{\xi_2,4}),\quad \varphi_3=\pazocal R_G^{(8)}(\theta_{\xi_3,8}),
$$ 
where the operators $\pazocal R_G^{(k)}$ are defined in \eqref{Reynolds}.
Associate to the $\varphi_i$ the $\Lambda$-invariant functions $\tilde\varphi_i$ on $\CC\times V$, as in \eqref{tilde}, and denote by $J$ the Jacobian $J(\tilde\varphi_i)$. Set $(t_0,v_0)=\left(1,\left(\frac18,\frac1{16},\frac14\right)\right)$. Then $J(t_0,v_0)\neq 0$.

\end{lemma}

\begin{proof}
Let us choose  $P_k=\left\{\tfrac\nu k\ | \ \nu\in\{0,1,\ldots,k-1\}^3\right\}$ for a set of representatives of $\frac 1k\ZZ^3/\ZZ^3$. We have
$$
\theta_{\xi_i,k_i}=\sum_{u\in\ZZ^3+\xi_i}q^{\frac{k_i}2 B\left[u\right]}e^{2k_i\pi i\trans vu}, \quad
i=0,1,2,3,\ (k_i)=(2,2,4,8).
$$
We compute the $G$-averages of these four theta functions, using formula \eqref{Reynolds}. By Proposition~\ref{U-matrices} and Theorem~\ref{TTF}, 
the elements $u^{(g,k_i)}_{m', m}$ ($m',m\in P_{k_i}$) of the matrices $U_g^{(k_i)}$ belong to the
cyclotomic field $\QQ\left(e^{\frac{\pi i}{k_i}}\right)$. We introduce the matrices of the Reynolds operators on the theta functions of degree $k_i$:
$$
\pazocal R_G^{(k)}=\left(r^{(k)}_{m', m}\right)_{m',m\in P_k},\quad r^{(k)}_{m', m}=\frac{1}{336}\sum_{g\in G}u^{(g,k)}_{m', m}, \quad
k=2,4,8.
$$
The exact values of the elements of the matrices $U_g^{(k_i)}$, belonging to $\QQ\left(e^{\frac{\pi i}{8}}\right)$, and the resulting Reynolds matrices were computed with Macaulay2; see~\cite{M2}.
We have 
$$
\varphi_i=\sum_{m\in P_{k_i}}r^{(k_i)}_{\xi_i, m}\theta_{m,k_i}
=\sum_{u\in\frac1{k_i}\ZZ^3}r^{(k_i)}_{\xi_i, u}q^{\frac{k_i}2 B\left[u\right]}e^{2k_i\pi i\trans vu},$$
where $r^{(k_i)}_{\xi_i, u}$ is defined to be $r^{(k_i)}_{\xi_i, m}$ for the unique $m\in P_{k_i}$ such that
$u\equiv m\mod\ZZ^3$. Replacing further $u$ with $\frac1{k_i}u$  with $u$ running over $\ZZ^3$ and differentiating,
we obtain 
\begin{gather*}
\pd\tilde\varphi_i/\pd t=
2k_it^{k_i-1}\sum_{u\in\ZZ^3}q^{\frac1{2k_i} B\left[u\right]}r^{(k_i)}_{\xi_i, \frac1ku}e^{2\pi i\trans vu},\\
\pd\tilde\varphi_i/\pd v_j=
2\pi it^{k_i}\sum_{u\in\ZZ^3}q^{\frac1{2k_i} B\left[u\right]}r^{(k_i)}_{\xi_i, \frac1ku}u_je^{2\pi i\trans vu},\quad
j=1,2,3.
\end{gather*}
As $B$ is positive definite, these infinite sums of powers of $q$ contain only a finite number of summands $N_i(c)$ with exponent at most $c$ for any given constant $c$, $N_i(c)\sim\frac{4\pi}{3\sqrt{\det B}}(2k_ic)^{\frac32}$, and, since $|q|<1$, the sums converge uniformly on compact subsets of $\CC\times V$. The convergence is in fact very rapid. For example, in the case of the slowest convergence, when $i=3, k_i=8$, if we stop summation at the terms of order $q^{7/2}\approx -2.3211\,i\cdot 10^{-13}$, then $N_3(3.5)=3527$, and all the vectors $u\in\ZZ^3$ occurring in the truncated sum are in the block $|u_1|\leq 10,|u_2|\leq 10,|u_3|\leq 14$. Computing the determinant of the thus obtained approximate Jacobian matrix at the point $(t_0,v_0)$, we obtain $J(t_0,v_0)\approx 0.000064967853+0.000075028580\, i$, all the shown decimal digits being exact.
\end{proof}

\section{Degree 8 hypersurfaces in \texorpdfstring{$\boldsymbol{\PP(1,1,2,4,7)}$}{P(1,1,2,4,7)} with correct singularities}
\label{sec:degree-8 hypersurfaces} 
The goal of this section is to show that a hypersurface in $\PP(1,1,2,4,7)$ defined by a \hbox{degree 8} homogeneous polynomial whose singularities are those of $\PP(1,2,4,7)$ is actually isomorphic to $\PP(1,2,4,7)$.  To achieve this, we  classify the degree $8$ hypersurfaces in $\PP(1,1,2,4,7)$ whose singularities are those of $\PP(1,2,4,7)$.

Recall that $\PP(1,2,4,7)$ embeds in $\PP(1,1,2,4,7)$ as the degree $8$ hypersurface given by the equation 
$$F_8^0=y_3^2-y_0y_4.$$
Its singular locus is the union of two irreducible components, $\PP^1=\ell$ and an isolated point $p$.  The singularity at $p$ is of analytic type $\frac17(1,2,4)$.  Here, for a cyclic group $\mU_d$ of order $d$, we denote by $\frac1d(\nu_1,\nu_2,\nu_3)$ the (analytic equivalence class of the) cyclic quotient singularity $\CC^3/\mU_d$, where the generator $c_d$ of $\mU_d$ acts by $c_d \colon (z_1,z_2,z_3)\mapsto (\epsilon^{\nu_1}z_1,\epsilon^{\nu_2}z_2,\epsilon^{\nu_3}z_3)$,\ \ $\epsilon=\exp\left(\frac{2\pi i}d\right)$.  At all but one point of $\ell$, the singularity of $X$ is of type $\frac12(1,0,1)$, that is, $ \CC\times A_1$, the Cartesian product of $\CC$ with a surface singularity $\CC^2/\langle-1\rangle$ of type $A_1$.  The type of the singularity at $q$, the unique point of $\ell$ where the type of singularity changes, is $\frac14(1,2,3)$.

Let us describe in affine charts the hypersurface $\{F_8^0=0\}$ in $\PP(1,1,2,4,7)$.  Write $y_0,\ldots,y_4$ for the coordinates in $\PP(1,1,2,4,7)$ with weights
$$n_0=1, \quad n_1=1, \quad n_2=2, \quad n_3=4, \quad n_4=7.$$
Set $y_i=1$, and quotient $\CC^4$ with coordinates $(y_0,\ldots,\widehat{y_i},\ldots,y_4)$ by the action of the cyclic group $\mU_{n_i}$ with weights $(n_0,\ldots,\widehat{n_i},\ldots,n_4)$ to obtain the affine chart $U_i$ of $\PP(1,1,2,4,7)$.  Note that $U_0$ and $U_1$ are smooth charts isomorphic to $\CC^4$ since $n_0=n_1=1$.

The restriction of $F_8^0$ to $U_0$ is $\{y_3^2-y_4=0\}$.  The hypersurface $F_8^0|_{U_0}=0$ is a smooth hypersurface parameterized by $(y_1,y_2,y_3)$, isomorphic to $\CC^3$.

The restriction of $F_8^0$ to $U_1$ is $y_3^2-y_0 y_4$.  The singular locus of $F_8^0|_{U_1}=0$ is the line $y_0=y_3=y_4$ with type $ \CC\times A_1$.

The restriction of $F_8^0$ to $U_2$ is $\{y_3^2-y_0 y_4=0\}/\frac{1}{2}(1,1,0,1)$.  Let us consider the quotient $\CC^3/\frac{1}{2}(1,0,1)$, where the coordinates of $\CC^3$ are denoted by $(u,v,w)$.  One can identify this quotient with the quadric $y_0 y_4 - y_3^2=0$ in $\CC^4$ through the mapping $(u,v,w) \mapsto (y_0=u^2,y_1=v,y_3=uw, y_4=w^2)$.  From this we observe that $\{y_0 y_4 - y_3^2=0\}/\frac{1}{2}(1,1,0,1)$ is the same as $\CC^3/\frac{1}{4}(1,2,3)$.  As a result, the restriction of $F_8^0$ to $U_2$ is isomorphic to $\CC^3/\frac{1}{4}(1,2,3)$.  Note that this is the affine chart $x_2=1$ in $\PP(1,2,4,7)$ with coordinates $(x_0,x_1,x_2,x_3)$.

The restriction of $F_8^0$ to $U_3$ is $\{y_0 y_4=1\}/\frac{1}{4}(1,1,2,3)$.  This chart is contained in $U_0 \cup U_4$.

The restriction of $F_8^0$ to $U_4$ is $\{y_3^2-y_0=0\}/\frac{1}{7}(1,1,2,4)$, which is isomorphic to $\CC^3/\frac{1}{7}(1,2,4)$, that is, the affine chart $x_3=1$ in $\PP(1,2,4,7)$.

Let us now consider a hypersurface $X$ in $\PP(1,1,2,4,7)$ defined by any degree $8$ homogeneous equation $F_8=0$ such that the singularities of $X$ are those of $\PP(1,2,4,7)$.

\begin{lemma} 
\label{lem:nonzero coefficients}
Write 
\begin{align*}
  F_8
=\; & 
y_4 \varphi_1(y_0,y_1) 
+ c_0 y_3^2 
+ c_1 y_2^2 y_3 
+c_2 y_2^4 
+ y_2 y_3 \varphi_2(y_0,y_1) 
+ y_2^3 \psi_2(y_0,y_1) &\\
&  
+ y_3 \varphi_4(y_0,y_1) 
+ y_2^2 \psi_4(y_0,y_1) 
+ y_2 \varphi_6 (y_0, y_1) 
+ \varphi_8(y_0,y_1),
\end{align*}
with $c_j \in \CC$ and $\varphi_i, \psi_i$ homogeneous polynomials of degree $i$.  Then both $c_0$ and $\varphi_1$ are non-zero.
\end{lemma}

\begin{proof}
First assume that $c_0=0$.  Then $Q_3=(0,0,0,1,0)$ is in $X$.  It is the origin of the chart $U_3=\CC^4/{\frac{1}{4}(1,1,2,3)}$ of $\PP(1,1,2,4,7)$.  Denote by $\tilde X_3 \subseteq \CC^4$ the hypersurface defined by $\{F_8|_{y_3=1}=0\}$, and let $\tilde Q_3 \in \tilde X_3$ be such that the image of $\tilde Q_3$ in the quotient $\CC^4/{\mU_4}$ is $Q_3$.  The germ of the canonical sheaf $\omega_{\tilde X_3,\tilde Q_3}$ is generated by
$$\sigma = {\rm res}_{\tilde X_3}\left( \dfrac{d y_0 \wedge d y_1 
\wedge d y_2 \wedge d y_4}{F_8|_{y_3=1}} \right),$$
of weight $-1$ for the $\mU_4$-action.  
So the Gorenstein index of $\tilde X_3/{\frac{1}{4}(1,1,2,3)} = X|_{U_3}$ 
in $Q_3$ is~$4$. 
But $X$ has no singularity of index $4$. 
We conclude that $c_0$ is non-zero.

Now assume that $\varphi_1=0$.  Set $X_4=X|_{U_4}$. We have $X_4 = \tilde X_4/\frac{1}{7}(1,1,2,4)$, where $\tilde X_4 \subset \CC^4$ is the hypersurface defined by the equation $\{F_8|_{y_4=1}=0\}$.  Under our assumption, $F_8$ does not contain $y_4$, so $F_8|_{y_4=1}=F_8$ defines the affine cone of an octic surface in $\PP(1,1,2,4)$ with coordinates $y_0,y_1,y_2,y_3$. The cone has to be smooth outside of its vertex $O$, the origin of the affine space $\CC^4$ with the same coordinates.  Indeed, assume the contrary. Then $O$ is a non-isolated singularity, so the image $Q_4$ of $O$ in $X_4$ also is a non-isolated singularity. Hence the unique isolated cyclic quotient singularity $p\in X$ of type $\frac17(1,2,4)$ is different from $Q_4$. Moreover, $p\in X$ is not a hypersurface singularity, but all the singularities of $X\setminus Q_4$ are hypersurface, so $p$ cannot be located in the chart $U_4$. Then, looking at the singularities of $X$ in the other charts, we see that all of them are hypersurface ones, except possibly points on the axis $(y_2,y_3)$. By proving that $c_0\neq 0$, we have excluded the possibility that $Q_3$ belongs to $X$, so the only non-hypersurface singularity that may occur in a chart $U_i$ with $i\neq 4$ is of the type (hypersurface in $\CC^4$)$/\frac12(1,1,0,1)$. The embedding dimension of such a singularity is at most that of $\CC^4/\frac12(1,1,0,1)$, which is equal to $7$, but the embedding dimension of $p$ is $12$. Hence no point of a chart $U_i$ for $i\neq 4$ can fit the role of $p$, so this case is impossible, and the singularity of $\tilde X_4$ at $O$ is an isolated quasi-homogeneous singularity.

It remains to see that under this assumption, the quotients $\{F_8=0\}/\frac{1}{7}(1,1,2,4)$ and $\CC^3/\frac{1}{7}(1,2,4)$ cannot be isomorphic. This follows from the fact that the first one is non-canonical, while the second one is canonical. Indeed, $\tilde X_4=\{F_8=0\}$ is non-canonical by Reid's criterion of canonicity (\textit{cf.} \cite[Theorem 4.1]{Reid}) with monomial valuation $\alpha$ defined by $(\alpha(y_i))=(\frac18,\frac18,\frac14,\frac12)$;
this implies that $X_4=\tilde X_4/\mu_7$ is non-canonical (\textit{cf.}~\cite[Proposition 1.7]{Reid}).
The canonicity of the second singularity follows from 
\cite[Theorem 3.1 and Remark 3.2]{Reid}. 
We conclude that $\varphi_1$ is non-zero.
\end{proof}

By Lemma~\ref{lem:nonzero coefficients}, 
one can assume that $c_0\neq 1$ and use the change 
of coordinates 
$$y_3 \longmapsto d_0 y_3 + d_1 y_2^2 + y_2 f_2(y_0,y_1) 
+ f_4(y_0,y_1),$$
with $d_j \in \CC$, $d_0 \in \CC^*$ 
and the $f_i$ homogeneous polynomials of degree $i$, to normalize the value of $c_0$ and to kill
the terms containing $y_3$ to the power of $1$.
We obtain for $F_8$ an expression as in the lemma, but with the constraints
$$c_0=-1,\quad c_1=0,\quad \varphi_2=0,\quad 
\varphi_4 = 0.$$ 

Next, again by the lemma, $\varphi_1\neq 0$, and we can change the variables $y_0, y_1$ in such a way that in new variables, we have $\varphi_1=y_0$, so that there is no monomial $y_1 y_4$ in the expression for $F_8$.  Then, we use the change of coordinates
$$y_4 \longmapsto y_4 + y_2^3 g_1(y_0,y_1) 
+ y_2^2 g_3(y_0,y_1) 
+ y_2 g_5(y_0,y_1) 
+ g_7(y_0,y_1),$$
with $g_i$ homogeneous of degree $i$, to kill all the monomials divisible by $y_0$.  As a result, we obtain an expression of the form
\begin{equation}\label{versal}
F_8 =  y_0 y_4 - y_3^2 
+ a_1 y_2^4 
+ a_2y_2^3  y_1^2 
+ a_3 y_2^2 y_1^4 
+ a_4 y_2 y_1^6 
+ a_5 y_1^8, \quad a_i \in \CC.
\end{equation}
Set $\nu = \min \{i \colon a_i \not=0\}$. 
Finally, using changes of coordinates 
of the form 
$y_1\mapsto \kappa y_1$, $y_2 \mapsto \lambda y_2 + \mu y_1^2$, 
$\kappa,\lambda \in \CC^*$, $\mu \in \CC$, 
we can reduce $F_8$ to one of the following normal forms:
\begin{equation}\label{nf}
\begin{array}{lllll}
&\nu =\infty\colon &  &F_8 = y_0 y_4 - y_3^2, &\\
&\nu =5 \colon &  &F_8 = y_0 y_4 - y_3^2 + y_1^8, &\\
&\nu =4 \colon &  &F_8 = y_0 y_4 - y_3^2 + y_2 y_1^6 , &\\
&\nu =3 \colon &  &F_8 = y_0 y_4 - y_3^2 + (y_2^2+a_5 y_1^4) y_1^4,& \\
&\nu =2 \colon &  &F_8 = y_0 y_4 - y_3^2 + (y_2^3+ 
a_4 y_2 y_1^4 + a_5 y_1^6)y_1^2, &\\
&\nu =1 \colon &  &F_8 = y_0 y_4 - y_3^2 
+ y_2^4+ (a_3 y_2^2  
+ a_4 y_2 y_1^2 + a_6 y_1^4) y_1^4. &
\end{array}
\end{equation}

\begin{proposition}
\label{pro:normal forms}
Except for the case $\nu=\infty$, the degree $8$ hypersurface $X$ defined by $F_8=0$ in $\PP(1,1,2,4,7)$, with $F_8$ one of the above normal forms, has a singularity of a type different from the types of singularities of $\PP(1,2,4,7)$.
\end{proposition}

\begin{proof}
Let $\nu=5$. 
Then $X$ passes through the origin of the chart $U_2$. We have
$U_2= \CC^4/\frac{1}{2}(1,1,0,1)$, and $X_2:=X\cap U_2$ is the quotient $\tilde X_2/\frac{1}{2}(1,1,0,1)$,
where $\tilde X_2 = \{y_0 y_4 - y_3^2 + y_1^8=0\} \subset \CC^4$.
Since $y_0 y_4 - y_3^2 + y_1^8$ is $\mU_2$-invariant and the differential 
$dy_0 \wedge d y_1 \wedge d y_3 \wedge d y_4$ is not, 
the generator 
$${\rm res}_{\tilde X_2}\left( \dfrac{d y_0 \wedge d y_1 
\wedge d y_3 \wedge d y_4}{y_0 y_4 - y_3^2 + y_1^8} \right)$$
of $\omega_{\tilde X_2,O}$, 
where $O$ is the origin of $\CC^4$, 
is anti-invariant under $\mU_2$. 
So $Q_2$, the image of $O$ in $U_2$, is a singular point of Gorenstein 
index $2$. 
But in $\PP(1,2,4,7)$, the only point of Gorenstein 
index $2$ 
is the origin of the chart $x_2=1$, with singularity $\frac{1}{4}(1,2,3)$. 
The latter singularity is non-isolated: the whole image of the coordinate axis with weight $2$ consists of singular points of the quotient. Hence it cannot be equivalent to the {\em isolated} singularity
$\{y_0 y_4 - y_3^2+y_1^8=0\}/{\tiny \frac{1}{2}(1,1,0,1)}$.
Thus the case $\nu=5$ is impossible.

The cases $2\leq \nu\leq 5$
are all treated in the same way: in all of them, we find an isolated singularity of Gorenstein index $2$ at the origin of the chart $U_2$, which is impossible.

\smallskip

Let us now look at the case $\nu=1$.  In the chart $U_2$, we have
$$X\cap {U_2} = \tilde X_2/\tfrac{1}{2}(1,1,0,1),$$
where $\tilde X_2 \subset \CC^4$  
is defined by the equation 
$\tilde F=0$ with $\tilde F=y_0 y_4 - y_3^2 + 1+a_3y_1^4+
a_4 y_1^6 + a_5 y_1^8$. 
The singular locus of $\tilde X_2$ 
is given by $y_0 = y_3 =y_4=0$, 
$y_1^2=\gamma$, where $\gamma$ is a double root of 
$a_5 t^4 +a_4 t^3 + a_3 t^2 +1 =0$. 
But this polynomial cannot have double roots. 
Indeed, otherwise $X$ would have singular points 
in the chart $U_1$, which are of the type of isolated hypersurface singularities:
$$X\cap U_1=\{y_0 y_4 - y_3^2 + y_2^4+ a_3y_2^2+
a_4 y_2  + a_5 =0\} \subset \CC^4.$$
This contradicts the fact that the only isolated Gorenstein singularity of $\PP(1,2,4,7)$ is of type $\frac{1}{7}(1,2,4)$, and this is not a hypersurface singularity.

Therefore, $\tilde X_2$ is smooth, and the singularities of $\tilde X_2/\mU_2 = X\cap {U_2}$ can only occur in a subset of the fixed locus of $\mU_2$. Thus $\Sing(X)\cap U_2$ is the image of the set
$$\{y_0=y_1=y_4=0\} \cap \tilde X_2 = \{y_0=y_1=y_4=-y_3^2+1=0\}.$$
Hence $X$ has two isolated singular points of type $\frac{1}{2}(1,1,1)$, which are $(0,0,\pm 1,0)$.  But this is impossible.
\end{proof}

This brings us to the main result of the paper. 

\begin{theorem}\label{main3}
The quotient $X=\JJJ/G$ of the Jacobian $\JJJ$ of the Klein quartic $C$ by its full automorphism group $G$ is isomorphic to the weighted projective space $\PP(1,2,4,7)$.
\end{theorem}
 
 \begin{proof}
This is an obvious consequence of Theorem~\ref{main2}, Lemma~\ref{lem:nonzero coefficients}, Proposition~\ref{pro:normal forms} and~\cite[Theorem 4.3]{G336-I}.
\end{proof}

\begin{remark}\label{smoothing}
By arguments similar to those used in the reduction of $F_8$ to normal forms \eqref{nf}, one can easily verify that in the $45$-dimensional group of coordinate changes in $\PP(1,1,2,4,7)$, the stabilizer of the octic form $F_8^0=y_0y_4-y_3^2$ is $13$-dimensional and the connected component of the identity in it is generated by the changes of the form
$$
y_0\mapsto \lambda_0 y_0,\ \ y_4\mapsto \lambda_0^{-1} y_4,
y_1\mapsto \mu y_0+\lambda_1 y_1, \ \ 
y_2\mapsto \lambda_2 y_2+h_2(y_0,y_1), \ \
$$
$$
y_3\mapsto y_3+y_0y_2h_1(y_0,y_1)+y_0h_3(y_0,y_1),
$$
$$
y_4\mapsto y_4+2(h_1(y_0,y_1)y_2+h_3(y_0,y_1))y_3+
y_0(h_1(y_0,y_1)y_2+h_3(y_0,y_1))^2,
$$ where $\lambda_i\in\CC^*$, $\mu \in \CC$ and the $h_i$ are homogeneous of degree $i$ in $y_0,y_1$.  Thus the dimension of the orbit of $F_8^0$ is $45-13=32$.  As the vector space of octic forms is of dimension $37$, the orbit of $F_8^0$ is of codimension $5$. A transversal slice to the orbit can be given by \eqref{versal}.  This family of hypersurfaces in $\PP(1,1,2,4,7)$ represents a versal deformation of the weighted projective space $\PP(1,2,4,7)$ inside the space of octics in $\PP(1,1,2,4,7)$.  The computation of the infinitesimal deformation space $\Ext^1(\Omega^1_X,\OOO_X)$ shows that $\dim T^1_X =5$, and this implies that \eqref{versal} is a complete versal deformation of $X$.  It is interesting to note that this family contains partial smoothings of $X$ with only three singular points, of which two are of type $\frac12(1,1,1)$ (the singular points found in the proof of the case $\nu=1$ of Lemma~\ref{pro:normal forms}) and the third one is $\frac17(1,2,4)$. These three singular points are non-smoothable, and even infinitesimally rigid by the result of Schlessinger~\cite{Schl}.
\end{remark}

\section{Relation to quotients by the Weyl group for the root systems \texorpdfstring{$\boldsymbol{B_3}$ and $\boldsymbol{C_3}$}{B\_3 and C\_3}}
\label{sec:various}

As we already noticed, $G$ contains the Weyl group $W=W(B_3)=W(C_3)$ of order $48$. We are going to look at the invariants under the action of $W$. For any $g\in W$, the symplectic matrix $\gamma_g=\left(\begin{smallmatrix}a&b\\c&d\end{smallmatrix}\right)$ is block diagonal; that is, $b=c=0$, and $a=(\trans{ d})^{-1}$ is nothing but the matrix of $g^{-1}$ in the basis given by the columns of $\omega_2=\alpha C$ (we keep the notation from previous sections). It easily follows from the definitions that $g\cdot\theta_{m,k}$ is just $\theta_{dm,k}$. This implies that the space of $W$-invariant theta functions of degree $k$ consists of the functions $\pazocal R_W\theta_{m,k}$, $m$ running over $(\trans{C})^{-1}(F)\cap P_k$, where $F$ is a fundamental domain for the action of the (real) affine crystallographic group $M\rtimes W$ on the space $M_\RR=M\otimes\RR$ containing the weight lattice $M$ and $\pazocal R_W$ is the Reynolds operator of taking the average over the action of $W$. When we say that $F$ is a fundamental domain, we mean that $F$ is the disjoint union of an open convex polyhedron in $M_\RR$ with finitely many polyhedra of smaller dimensions contained in its boundary such that every $W$-orbit in $M_\RR$ has exactly one representative in $F$. We are now going to fix the choice of a particular fundamental domain $F$.

There is a tower of degree $2$ extensions
$$
Q=Q(C_3)\subset Q\dual=Q(B_3)\subset M=Q^*,
$$
the three lattices being invariant under $W$, so taking the semi-direct product with $W$, we obtain a tower of degree $2$ extensions of the corresponding real affine crystallographic groups:
$$
\tilde W=Q\rtimes W\subset \tilde W\dual=Q\dual\rtimes W\subset \tilde W^* =M\rtimes W.
$$
Both $\tilde W$ and $\tilde W\dual$ are affine Weyl groups and are generated by affine reflections. By~\cite[Section~VI.2.2]{Bou}, they have fundamental domains which are closed tetrahedra, called alcoves. For $\tilde W$, the vertices of the standard alcove $\tilde C=C_{\tilde W}$ are $f_0=0$, $f_1=(1,0,0)$, $f_2=\frac12(1,1,0)$, $f_3=\frac12(1,1,1)$. Denoting by $x_1,x_2,x_3$ the coordinates in the Euclidean space $\RR^3$, in which we place our lattices $Q$ and $M$, we obtain an alcove $\tilde C\dual=C_{\tilde W\dual}$ for $\tilde W\dual$ as one half of $\tilde C$ cut out by the mirror $x_1=\frac12$ of a reflection contained in $\tilde W\dual$ but not in $\tilde W$. Thus $\tilde C\dual$ is the tetrahedron with vertices $O=f_0$, $A=\frac12f_1$, $E=f_2$, $N=f_3$. The subgroup of $\tilde W^*$ leaving invariant the tetrahedron $OAEN$ is of order $2$; besides the identity, it contains an element $R$ of order $2$, the axial symmetry, or rotation by the angle $\pi$ with axis $KL$, where $K$ and $L$ are the middles of $AE$ and $ON$, respectively (see the figure below).

Now we can get a fundamental domain $F$ of $\tilde W^*$ as follows: first choose a plane $P$ containing $KL$, and pick up one of the two halves in which $P$ dissects $\tilde C\dual$, say $C_1$. The intersection $D=P\cap\tilde C\dual$ is a face of $C_1$, and it is dissected by $KL$ in two halves $D_1$ and $D_2$ symmetric to each other by the action of $R$. Then we obtain the following set of representatives
of the orbits of $R$ acting on $\tilde C\dual$: the polyhedron $C_1$, in which we include all of its faces, except for the face $D$, from which only the closed part $D_1$ is included. The representatives of the orbits of $R$ in $\tilde C\dual$ are at the same time the representatives of the orbits of $\tilde W^*$ in $V^*_\RR$, so the constructible set we have described is a fundamental domain for $\tilde W^*$.


\[
\includegraphics{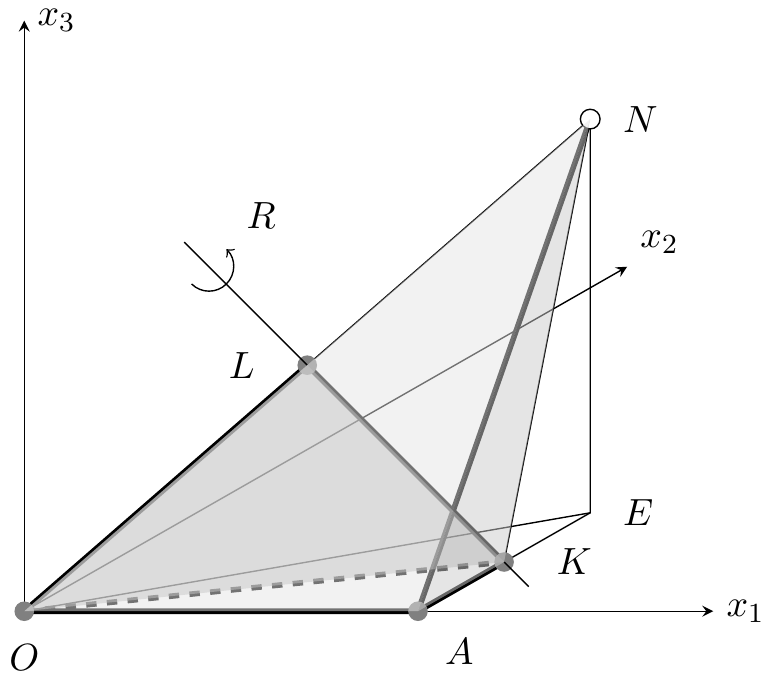}
\]

To obtain a particular fundamental domain $F$, we choose in the above construction the plane $OKN$ for $P$,  the tetrahedron $OAKN$ for $C_1$, then the triangle $KON$ for the face $D$, and we keep only one half of it in~$F$, the triangle $D_1=OKL$. On the picture, $F$ is shaded; the triangle $OKL$ is shaded in a darker gray color as it is included in $F$ and the remaining part of the face $OKN$ is not included. The sides of the three closed triangles of the boundary of $F$ contained in $F$, as well as the edge  $AN$ contained in $F$, are drawn as thick lines. Only the part $OL$ of the edge $ON$ drawn as a thick line is contained in $F$, and the edge  $KN$, not contained in $F$, is drawn as a thin line. The black points $O$, $A$, $K$, $L$ are contained in $F$, while the vertex $N$ is not in $F$ and is represented by a white point.

Looijenga in~\cite{Loo} provides a recipe for the choice of generators of the algebra of $W$-invariant theta functions for the lattice of the form $\tau Q+Q$; his generators are the averages of the theta functions $\theta_{f_i,k_i}$ with characteristics $f_i$, the vertices of the alcove, the degrees $k_i$ being the smallest positive integers such that $k_if_i\in M$. 

\begin{remark}\label{conj-gen}
In the previous sections, we worked with theta functions $\theta_{m,k}$ defined in Section~\ref{sec:prelim} for a different lattice, $\tau M+Q$, so Looijenga's recipe does not apply to $\theta_{m,k}$, but we can use it as a heuristic principle for the choice of theta functions that would be candidates for generators of the $G$-invariant subalgebra. We can try the $G$-averages of the theta functions whose characteristics are the vertices $\mu_i$ of a ``pseudo-fundamental domain'' $F_G\subset F$ for $G$, or a ``pseudo-alcove'' which is just a smaller tetrahedron with vertex $O$ contained in~$F$, having the correct lattice volume with respect to $M$, equal to $\frac1{336}$, and the expected weights $k_i$.
Recall that $b_1=f_1,b_2=2f_2,b_3=f_3$ are the columns of the matrix $(\trans{C})^{-1}$; that is, they form a basis of $M$ dual to the standard basis of the root system $C_3$. This is the basis in which we write the characteristics of our theta functions. From this we see, in particular, that the lattice volume of $F$ with respect to $M$ is $\frac1{48}$. We can pick up the following \textit{ad hoc} candidate for the ``pseudo-alcove'': take the tetrahedron $F_G$ with vertices
\begin{equation}\label{vertices}
\mu_0=0,\quad 
\mu_1=\tfrac12b_1, \quad 
\mu_2=\tfrac14(b_1+b_2), \quad 
\mu_3=\tfrac17(2b_1+b_2+b_3).
\end{equation}
The tetrahedron $F_G$ obviously has $M$-volume $\frac1{336}$,
and the smallest $k_i$ such that $k_i\mu_i\in M$ coincide with the weights of the expected weighted projective space: $(k_0,k_1,k_2,k_3)=(1,2,4,7)$. We thus may expect that the algebra $S(\LLL^2)^G$ is generated by  the $G$-averages  of the even-degree theta functions that are obtained as products of theta functions whose  characteristics are the vertices of $F_G$: 
 $\theta_{\mu_0,1}^2$, $\theta_{\mu_1,2}$, $\theta_{\mu_2,4}$,
$\theta_{\mu_0,1}\theta_{\mu_3,7}$, $\theta_{\mu_3,7}^2$. To present four algebraically independent $G$-invariant theta functions of degrees $2,2,4,8$ in Lemma~\ref{Jac-non-0}, we chose another set, for which the computations are slightly easier: $\theta_{\mu_0,2}$, $\theta_{\mu_1,2}$, $\theta_{\mu_2,4}$, $\theta_{\frac78\mu_3,8}$.
\end{remark}

We now return to the quotient $\JJJ/W$. As we saw above, the theta functions
$\{\pazocal R_W\theta_{m,k}\}_{m\in (\trans{C})^{-1}(F)\cap P_k}$ form a basis of
the $W$-invariant theta functions of degree $k$ on $\JJJ$. The fact that $\JJJ/W$ is not a complex crystallographic reflection quotient manifests itself in that $F$ is not a closed simplex but a constructible set, which we can describe as follows:
\begin{equation}\label{FW}
F=OAKN\cup\bar{OAK}\cup\bar{OKL}\cup\bar{OAL}\cup ANK\cup ANL\cup AN.
\end{equation}
Here we denote by $\bar{A_1\ldots A_n}$ the convex hull of points $A_1, \ldots, A_n$ in a real affine space, and by $A_1\ldots A_n$ its relative interior. From this we can deduce the Hilbert function of the invariant algebra $S(\LLL)^W$. 

\begin{prop}\label{hTW}
The Hilbert function $h_{\bar S}\colon k\mapsto \dim H^0(\JJJ,\LLL^k)^W$ of the algebra $\bar S=S(\LLL)^W$
is the sum of a polynomial $E$ and a function $\ell\colon k\mapsto d_0k+d_1$, where $d_0,d_1$ are $4$-periodic functions of the integer variable $k$,
$$
E(k)= \tfrac{1}{48}k^3+ \tfrac{3}{16}k^2 +\tfrac{2}{3}k+ {1}, \quad d_0=
\left\{\begin{array}{l}
0\ \mathrm{if}\  k\equiv 0\ \mathrm{or}\ 2\\
-\tfrac{3}{16}\ \mathrm{if}\  k\equiv \pm1
\end{array}\right., \quad 
d_1=\left\{\begin{array}{l}
0\ \mathrm{if}\  k\equiv 0\\
-\tfrac{11}{16}\ \mathrm{if}\  k\equiv \pm1\\
-\tfrac14\ \mathrm{if}\  k\equiv 2
\end{array}\right.
\ \ (\mod 4).
$$

The Hilbert series $\HS_{\bar S}(t)=\sum_{k=0}^\infty h_{\bar S}(k)t^k$ of\, $\bar S$ is given by
\begin{equation}\label{hSer}
\HS_{\bar S}(t)=\frac{1-t+t^2}{(1-t)^2(1-t^2)(1-t^4)}.
\end{equation}
\end{prop}

\begin{proof}
By Ehrhart's theorem, the number of integer points $h_P(k)$ in the integer multiples $kP$ of an open or closed polytope $P$ in $\RR^n$ with rational vertices is a quasi-polynomial in the integer variable $k$ of degree equal to the dimension of $P$. A quasi-polynomial is a polynomial function whose coefficients are periodic. The common period of the coefficients of the Ehrhart quasi-polynomial of $P$ is the smallest positive integer $d$ such that $dP$ is a lattice polytope; that is, all the coordinates of its vertices are integers (see, for example, \cite[Theorem 4.6.25]{Sta}). By the above, $h_{\bar S}$ coincides with the function $h_F$ counting the number of points of $M$ in the multiples $kF$ of $F$.
By \eqref{FW}, $h_F(k)$ can be represented as a linear combination, with coefficients $\pm 1$, of the numbers $h_\sigma(k)$ for $\sigma$ running through a finite set of open or closed simplices with rational vertices of dimensions between $0$ and $3$:
$$
h_F=h_{OAKN}+h_{\bar{OAK}}+h_{\bar{OKL}}+h_{\bar{OAL}}+h_{ALN}+h_{ANK}
-h_{\bar{OA}}-h_{\bar{OK}}-h_{\bar{OL}}+h_{AN}+h_O.
$$
Each of the terms of this linear combination is a quasi-polynomial of degree at most $3$ with period $d=4$, the least common denominator of the coordinates of the vertices of the simplices $\sigma_i$, so $h_F$ also is such a quasi-polynomial. Moreover, the dominant coefficient of $h_F$ is constant and is equal to the inverse of the lattice volume of $F$, that is, $\frac1{48}$, so it suffices to compute 12 consecutive values of $h_F$ in order to determine the coefficients as solutions to a system of linear equations. Here is the result of the computation of $h_{\bar S}(k)=h_F(k)$ for $0\leq k \leq 12$ by Macaulay2~\cite{M2}:
$$1,\ 1,\ 3,\ 4,\ 8,\ 10,\ 16,\ 20,\ 29,\ 35,\ 47,\ 56,\ 72,\ldots .$$
These values  completely determine $h_{\bar S}$, and we thus obtain the formulas from the statement of the proposition.
\end{proof}

\begin{prop}\label{fake}
Consider the quotient $X'=\PP^3/\mU_2\times\mU_4$ of the projective space $\PP^3$ by a group of order $8$, where $\mU_n$ denotes a cyclic group of order $n$ and the generators of $\mU_2$, $\mU_4$ act by diagonal matrices $\diag (1,1,-1,1)$, $\diag(i,1,1,-1)$, respectively. Then $X'$ and $\JJJ/W$ have the same Hilbert functions.
\end{prop}

\begin{proof}
We represent $X'$ as a toric variety, the equivariant compactification $\mathbf X_{\Sigma, \mathsf{N}}$ of the $3$-dimensional algebraic torus $\TT:=(\CC^*)^3$, defined by a fan $\Sigma$ in the $3$-dimensional $\RR$-vector space $\mathsf{N}\otimes \RR$, where $\mathsf{N}\simeq\ZZ^3$ is the lattice of $1$-parametric subgroups in $\TT$ 
(see \textit{e.g.}~\cite{Fu} for definitions and basic properties of toric varieties). Let $y_0,\ldots,y_3$ be the homogeneous coordinates of $\PP^3$ and $\TT_0=(\CC^*)^3$ the standard torus in the affine chart $\CC^3$ of $\PP^3$ with affine coordinates $u_1=y_1/y_0,\ u_2=y_2/y_0,\ u_3=y_3/y_0$. The simplicial fan $\Sigma$ defining $\PP^3$ is determined by its $1$-dimensional cones, or rays, which are spanned by
the four vectors $(1,0,0),\ (0,1,0),\ (0,0,1), (-1,-1,-1)$ of the lattice $\mathsf{N}_0=\ZZ^3$. The dual to $\mathsf{N}_0$ is the lattice $\mathsf{M}_0=\ZZ^3$ of exponents of monomials in the coordinate algebra $\CC[\mathsf{M}_0]=\CC[u_1,u_2,u_3]$ of the affine chart of $\PP^3$ that we have chosen. To get the quotient by $\mU_2\times\mU_4$, we replace $\mathsf{M}_0$ with the sublattice $\mathsf{M}=\mathsf{M}_0^{\mU_2\times\mU_4}$ of exponents of  $(\mU_2\times\mU_4)$-invariant monomials, which is given by 
$$
\mathsf{M}=\{(m_1,m_2,m_3)\in \mathsf{M}_0\colon m_3\equiv 0\mod 2, \ -m_1 - m_2 +m_3\equiv 0\mod 4\};
$$
this is the lattice of exponents of the monomials $\mathbf u^{\mathbf m}=u_1^{m_1}u_2^{m_2}u_3^{m_3}$
which form a basis of the regular functions on the quotient torus $\TT:=\TT_0/\mU_2\times\mU_4$.
The dual $\mathsf{N}=\mathsf{M}^*=\left(\mathsf{M}_0^{\mU_2\times\mU_4}\right)^*$ is the overlattice of $\mathsf{N}_0=\ZZ^3$ generated by the vectors $(1,0,0),
\frac12(0,1,0), \frac14(-1,-1,1)$, and $X'=\mathbf X_{\Sigma, \mathsf{N}}$ is the equivariant compactification of $\TT$ defined by the same fan $\Sigma$ as the original $\PP^3$, but taken with respect to the lattice $\mathsf{N}$. 
The primitive vectors of $\mathsf{N}$ generating the rays of the fan are $v_0=(1,0,0)$,
$v_1=\frac12(0,1,0)$, $v_2=-\frac12(1,1,1)$, $v_3=(0,0,1)$, and we get four $\TT$-invariant divisors $D_i$ in $X'$, defined by
$D_i=D_{v_i}$, where $D_{v_i}$ denotes the closure in $X'$ of the kernel of $v_i$ viewed as a $1$-parametric subgroup of $\TT$. The Cartier indices of these divisors for $i=0,1,2,3$ are, respectively, $4,2,2,4$. We omit the routine verification of the following assertion, which makes more precise the statement we are proving.
\end{proof}

\addtocounter{theorem}{-1}
\renewcommand{\thetheorem}{\arabic{section}.\arabic{theorem}$'$}

\begin{prop}\label{D03}
In the above notation, the Hilbert functions $k\mapsto h^0(X',\OOO_{X'}(kD_i))$ coincide with $h_{\bar S}$ for both divisors $D_i$ of Cartier index $4$, that is, for $i=0$ and $3$.
\end{prop}

\renewcommand{\thetheorem}{\arabic{section}.\arabic{theorem}}

It is plausible that the two quotients are indeed isomorphic. We observe that $X'$ belongs to the class of projective toric varieties of dimension $n$ whose fan contains $n+1$ rays. Such varieties are called weak weighted projective spaces; 
some authors call them fake weighted projective space (see~\cite{Kas}), but we prefer the adjective weak because the class of weak weighted projective spaces contains all the weighted projective spaces. A weak weighted projective space is a weighted projective space if and only if the primitive vectors of $\mathsf{N}$ on the rays $\RR^*_+v_i$ of its fan generate the whole of $\mathsf{N}$. One easily sees that this is not the case for $X'$. This brings us to the following generalization of the Bernstein--Schwarzman conjecture.

\begin{conj}\label{WWPS}
If\, $\Gamma$ and $\Gamma_1$ are commensurable complex crystallographic groups acting on $\CC^n$ such that $\ud\Gamma_1=\ud\Gamma$, and if\, $\Gamma$ is irreducible and generated by reflections,  then the quotient $\CC^n/\Gamma_1$ is a weak weighted projective space. It is a genuine weighted projective space if and only if\, $\Gamma_1$ is also generated by reflections.
\end{conj}

The cases when $\Gamma$ is one of the complex crystallographic reflection groups $(Q+\tau Q)\rtimes W$ or
$(Q+\tau Q\dual)\rtimes W$
treated by Looijenga and Bernstein--Schwarzman and
$\Gamma_1=\Lambda\rtimes W$, where $\Lambda=Q+\tau M$, are particular cases of this conjecture: while 
$\CC^3/\Gamma$ is known to be a weighted projective space, we expect that $\CC^3/\Gamma_1$ is the weak weighted projective space $X'$.


\end{document}